\documentclass[11pt,times,letter]{article}

\usepackage{amsmath, amsfonts,amsthm,amssymb,bm}
\usepackage{dsfont}

\usepackage[english]{babel}
\usepackage{float}

\usepackage[normalem]{ulem}

\usepackage{subfigure}
\usepackage{latexsym,amssymb,amsfonts,graphicx}
\usepackage{amsmath,dsfont}
\usepackage{verbatim}
\usepackage{mathrsfs}
\usepackage{bm}
\usepackage{color}
\usepackage{epsfig}
\usepackage{epstopdf}
\usepackage[title]{appendix}
\usepackage{url}
\usepackage{bm}
\usepackage{natbib}
\usepackage{tikz}
\usetikzlibrary{arrows.meta, positioning}

\usepackage[colorlinks,linkcolor=blue,citecolor=blue,anchorcolor=blue]{hyperref}

\newcommand{\Px}{ \mathbb{P} }

\newcommand{\Ex}{ \mathbb{E} }

\def\esssup_#1{\underset{#1}{\mathrm{ess\,sup\, }}}
\def\essinf_#1{\underset{#1}{\mathrm{ess\,inf\, }}}
\def\argmax_#1{\underset{#1}{\mathrm{arg\,max\, }}}
\def\argmin_#1{\underset{#1}{\mathrm{arg\,min\, }}}

\newcommand{\Fx}{\mathbb{F} }

\newcommand{\R}{\mathds{R}}

\newtheorem{theorem}{Theorem}[section]

\numberwithin{equation}{section}

\newtheorem{proposition}[theorem]{Proposition}
\newtheorem{remark}[theorem]{Remark}
\newtheorem{lemma}[theorem]{Lemma}
\newtheorem{corollary}[theorem]{Corollary}

\allowdisplaybreaks

\definecolor{Red}{rgb}{1.00, 0.00, 0.00}

\definecolor{DRed}{rgb}{0.5, 0.00, 0.00}

\definecolor{Blue}{rgb}{0.00, 0.00, 1.00}

\definecolor{Green}{rgb}{0.0, 0.4, 0.0}

\title{Optimal consumption under relaxed benchmark tracking and consumption drawdown constraint}

\author{
Lijun Bo \thanks{Email: lijunbo@ustc.edu.cn,  School of Mathematics and Statistics, Xidian University, Xi'an, 710126, China.}
\and
Yijie Huang \thanks{Email: huang1@mail.ustc.edu.cn, Department of Applied Mathematics, The Hong Kong Polytechnic University, Hung Hom, Kowloon, Hong Kong.}
\and
Kaixin Yan\thanks{Email: kaixinyan@stu.xmu.edu.cn, School of Mathematical Sciences, Xiamen University, Xiamen 361005, China.}
\and
Xiang Yu \thanks{Email: xiang.yu@polyu.edu.hk, Department of Applied Mathematics, The Hong Kong Polytechnic University, Hung Hom, Kowloon, Hong Kong.}
}

\date{\vspace{-0.5in}}

\begin{document}
\maketitle

\begin{abstract}

This paper studies an optimal consumption problem with both relaxed benchmark tracking and consumption drawdown constraint, leading to a stochastic control problem with dynamical state-control constraints. In our relaxed tracking formulation, it is assumed that the fund manager can strategically inject capital to the fund account such that the total capital process always outperforms the benchmark process, which is described by a geometric Brownian motion. We first transform the original regular-singular control problem with state-control constraints into an equivalent regular control problem with a reflected state process and consumption drawdown constraint. By utilizing the dual transform and the optimal consumption behavior, we then turn to study the linear dual PDE with both Neumann boundary condition and free boundary condition in a piecewise manner across different regions. Using the smooth-fit principle and the super-contact condition, we derive the closed-form solution of the dual PDE, and obtain the optimal investment and consumption in feedback form. We then prove the verification theorem on optimality by some novel arguments with the aid of an auxiliary reflected dual process and some technical estimates. Some numerical examples and financial insights are also presented. 

\vspace{0.1in}

\end{abstract}

\section{Introduction}\label{sec:intro}

In the wake of Merton's pioneer studies in \cite{Merton69,Merton1971}, the pursuit of optimal decision making in portfolio management and consumption plan via utility maximization has prompted significant growth. Theoretical enhancements have been developed to confront an array of emerging challenges originating from intricate market models, advanced performance metrics,  state and/or control constraints, and other pertinent aspects.

One burgeoning direction to generalize Merton's problem focuses on the influence of the past consumption peak on the current consumption plan. Large expenditures may psychologically lift up the agent's standard of living, thereby affecting the expected utility. In the seminal work, \cite{Dybvig95} formulates an infinite horizon utility maximization problem under ratcheting constraint on consumption $c_t\geq \sup_{s\leq t}c_s$, premising the situation with non-decreasing consumption rate.   \cite{Arun12} generalizes \cite{Dybvig95}'s consumption ratcheting model through a drawdown constraint requiring $c_t \geq \lambda \sup_{s\leq t}c_s$ with $\lambda\in[0,1]$. Unlike ratcheting—which rigidly prevents any consumption reduction (an extreme habit formation)—the drawdown constraint only adheres to a proportion of the past consumption peak, offering more flexibility in the path-dependent impact. Along this direction, fruitful research studies can be found by considering different variations and extensions. To name a few, \cite{Bahman2019} study a similar utility maximization problem under a drawdown constraint on the excessive dividend rate until the bankruptcy time by mandating the current consumption rate to stay above a fraction of the past consumption peak $c_t\geq \lambda \sup_{s\leq t}c_s$ with $\lambda\in[0,1]$. \cite{G2020} propose the shortfall averse preference to measure the performance of relative consumption on the ratio between the current consumption and the historical consumption running maximum process. \cite{JeonPark21} extend \cite{Bahman2019}'s analysis by employing the martingale duality method to reformulate the control problem as an infinite-horizon optimal stopping problem under a general utility function. Subsequently, \cite{JO22} generalize the approach in \cite{JeonPark21} to the finite-horizon case. \cite{DLPY22} propose a utility formulation based on the difference between current consumption and the historical consumption running maximum, deriving a closed-form optimal consumption policy that varies piecewisely with wealth level. \cite{Alb22, Alb23} study an optimal dividend problem under ratcheting and drawdown constraints, respectively. \cite{LYZ} extend the framework of \cite{G2020} by incorporating life insurance purchases under shortfall-averse preferences, adding a consumption drawdown constraint. Meanwhile, \cite{LYZ24} generalize the formulation of \cite{DLPY22} under S-shaped utility, explicitly modeling the loss aversion effect in relative consumption. Further extensions include that \cite{Liang23} incorporate both reference-level and drawdown constraints into the \cite{DLPY22}'s  framework, uncovering new financial implications.
\cite{Tanana23} refines the convex duality approach to derive a general duality theorem for optimal consumption under drawdown constraints in incomplete semimartingale markets. Recent work by \cite{CLYY23} revisits the finite-horizon problem of \cite{JeonPark21} on excessive consumption under drawdown constraints. Using PDE techniques, they characterize the associated time-dependent free boundaries in closed form.

Another important branch of research in optimal investment and consumption is concerned with the performance relative to an exogenous benchmark process, which might refer to the market index process, inflation rates, liabilities, etc.  Portfolio management problems with various types of benchmark tracking have been extensively studied over the past decades. \cite{Browne9} initiates the study of active portfolio management by maximizing the probability of reaching a target wealth level while outperforming a benchmark. Subsequently, \cite{Browne99} addresses the dual problem of minimizing the expected time to achieve a performance goal. \cite{Browne00} further generalizes these objectives by combining the expected reward maximization upon achieving the goal and the expected penalty minimization upon falling to a shortfall level.
A parallel strand of literature focuses on tracking error minimization, often formulated as linear-quadratic stochastic control problems (\citealt{Gaivoronski05, YaoZZ06, NLFC}).

In sharp contrast, \cite{BoLiaoYu21} propose a tracking framework based on fictitious capital injection, in which it is presumed that a fund manager may tactically infuse fictitious capital as singular control into the fund account such that the total capital surpasses the designated benchmark process. More recently, \cite{BHY23a,BHY23b} generalize this tracking formulation to accommodate the consumption plan via utility maximization. In the present paper, we further combine this optimal tracking control problem with both dynamic state constraints and  the additional drawdown constraint on the consumption control (see \eqref{eq:w}). Our imposition of a drawdown constraint on consumption is grounded in habit formation model (\citealt{Constantinides1990, Detemple91, Detemple92, Campbell1999}): it prevents investors, once accustomed to a certain standard of living, from facing severe deteriorations in their consumption level. Unlike the ratcheting constraint as in \cite{Dybvig95}, our formulation with drawdown constraint allows for managed habit adjustment, striking a balance between strict habit persistence and unconstrained models. In particular, we accommodate scenarios where agents accept moderate consumption cuts during adverse market conditions—a feature that aligns more closely with empirical behavior. Additionally, we investigate how this drawdown constraint influences both consumption behavior and capital injection strategies. A key distinction from classical models (e.g., \citealt{Arun12}, \citealt{JeonPark21} and \citealt{JO22}) is that previous studies often require a threshold of the initial wealth to support the drawdown constraint. In contrast, our formulation ensures the feasibility of the drawdown and benchmark constraints through the strategic capital injection, eliminating the need of the restrictive initial wealth condition.

 Our formulation yields a class of non-standard optimal control problems involving both regular and singular controls, subject to state-control constraints. Note that one cannot apply the conventional dimension-reduction technique or variable transformations here, as the objective functional in problem \eqref{eq:w} lacks the desired homogeneity. To maintain tractability within a Markovian framework, the value function must incorporate three state variables, further complicating the analysis. To address the dynamic state-control constraint in problem \eqref{eq:w}, we first adapt the methodology from \cite{BHY23a,BHY23b} to reformulate it as an equivalent problem \eqref{eq:u0} involving only regular controls. This transformation allows us to incorporate the wealth constraints (induced by benchmark tracking) through a reflected state process at the boundary zero. The key distinction from \cite{BHY23a,BHY23b} lies in the additional challenge posed by the drawdown constraint on consumption in this auxiliary control problem. Specifically, we must handle not only the reflected state process but also an extra state variable: the running maximum of consumption, which complicates the analysis.

Mathematically speaking, we study an associated Hamilton-Jacobi-Bellman variational inequality (HJB-VI) featuring a Neumann boundary condition (arising from the reflected state process) and
a free boundary condition (induced by the consumption drawdown constraint). The primary contributions of this work are (i) a rigorous analysis of this novel three-dimensional HJB equation with mixed boundary conditions; (ii) a technical verification for optimality, addressing two key challenges:  
the objective functional in \eqref{eq:u0} incorporates the local time of the reflected state process (a non-standard feature) and the drawdown constraint introduces path-dependent dynamics. To tackle the HJB-VI, we decompose the three-dimensional domain into five regions and reformulate the problem as a piecewise linear dual PDE with Neumann/free boundary conditions and a super-contact condition. By conjecturing a separation-form solution for the dual PDE. Leveraging the smooth-fit principle and super-contact condition, we derive (i) a closed-form solution for the dual PDE; (ii) an explicit characterization of the free boundary as the unique solution to an algebraic equation (Proposition \ref{prop:sol-v}); (iii) the primal solution via an inverse transform, expressing the value function and optimal feedback policies (investment/consumption) in terms of the original state variables.

We highlight some key distinctions between our verification arguments for Theorem \ref{thm:verification} in this paper and that in the existing literature. Contrary to \cite{BHY23a,BHY23b}, the drawdown constraint of consumption complicates the feedback functions of optimal portfolio and consumption and introduces a free boundary implicitly depending on the historical maximum consumption level. The technique in \cite{BHY23a,BHY23b}—relying on the dual representations and estimates of feedback controls—can not be directly applied, particularly for establishing the transversality condition. To resolve these issues, the current paper provides some analytical properties of the free boundary (Proposition \ref{prop:sol-v} and Lemma \ref{lem:boundary-m}) and derives the linear growth of feedback controls across regions (Lemma \ref{lem:feedback-control}), ensuring the feasibility of optimal strategies and the well-posedness of the reflected SDE of $X^*=(X_t^*)_{t\geq0}$  under the optimal (feedback) policy. On the other hand, prior studies on drawdown constraints and consumption running maxima—such as those by \cite{Bahman2019}, \cite{DLPY22}, and \cite{LYZ24}--rely on the standard Black-Scholes model, where the dual process (i.e., the state price density) can be used to simplify the verification of the transversality condition. However, we focus on an auxiliary control problem governed by a non-standard reflected SDE (see \eqref{state-X}). This setting presents a key challenge: the absence of a well-established duality theory to overcome some obstacles in proving the verification theorem. In response, we introduce an auxiliary dual reflected diffusion process (defined in \eqref{eq:Yt}), constructed directly from the primal reflected optimal control model. Leveraging this dual process, we derive a duality representation of the primal value function (see \eqref{eq:v-Yt}) and the duality inequality (see \eqref{eq:hatv-y}). Second, leveraging the duality inequality, we develop a new verification argument that transfer the transversality condition from the primal state processes to this reflected dual process. Lastly, building on the newly established properties and moment estimates of the reflected dual process (Lemmas \ref{lem:Yt} and \ref{lem:transversality}), we conclude the desired transversality conditions under a mild condition on the discount rate. As shown in Theorem \ref{thm:verification}, we stress that the optimal portfolio and consumption control processes exhibit path dependence on both the wealth process $V^{\theta,c}=(V^{\theta,c}_t)_{t\geq0}$ and the benchmark process $Z=(Z_t)_{t\geq0}$, rendering the direct approach based on the original state process $V^{\theta,c}=(V^{\theta,c}_t)_{t\geq0}$ intractable. The introduction of the auxiliary reflected state process $X=(X_t)_{t\geq0}$ is necessary, which enables a feedback form of the optimal strategies (see Remark \ref{rem:XV}). The overall methodology—combining the auxiliary problem with the dual transformation—is outlined in the flow chart below.

Furthermore, leveraging the reflected dual process, we prove in Lemma \ref{lem:inject-captial} that the expected discounted total capital injection remains finite and strictly positive. This ensures the necessity of capital injection to satisfy the dynamic benchmark tracking constraint, while also guaranteeing that the problem remains well-posed—ruling out unrealistic scenarios requiring infinite capital to meet both the benchmark tracking and drawdown consumption constraints. To complement the theoretical analysis, we present some numerical examples on the sensitivity of the optimal feedback controls and the expected discounted capital injection with respect to model parameters. These plots and simulations help to draw some financial insights from our theoretical formulas. For instance, when capital injection is triggered to satisfy the benchmark tracking constraint, the consumption is often observed to stay at its lower bound, $\lambda \sup_{s\leq t}c_s$, implying that the peak consumption $\sup_{s\leq t}c_s$ remains unchanged. Intuitively, when the wealth is sufficiently high, the optimal consumption under a drawdown constraint ($\lambda > 0$) is surprisingly lower than in the unconstrained case ($\lambda = 0$). This suggests that the drawdown constraint discourages aggressive spending in high-wealth scenarios, as elevated consumption would raise the reference process $\sup_{s\leq t}c_s$, tightening future consumption limits. Conversely, the optimal portfolio allocation under the drawdown constraint is larger, requiring the more aggressive investment to sustain the constraint. In addition, a stricter drawdown also necessitates higher capital injections to support the elevated consumption threshold. 

The remainder of this paper is organized as follows. Section \ref{sec:model} introduces the optimal consumption problem with relaxed benchmark tracking under the consumption drawdown constraint. We reformulate the original problem into an equivalent auxiliary control problem using a reflected state process, then derive the associated HJB-VI with a Neumann boundary condition. In Section \ref{sec:solvHJBVI}, we decompose the HJB-VI into piecewise segments based on optimal consumption behavior. Employing the smooth-fit principle and super-contact conditions, we obtain a closed-form solution to the dual PDE under Neumann and free boundary conditions. Section \ref{sec:verification} presents the verification theorem and characterizes the optimal feedback investment and consumption strategies in piecewise form. Numerical examples and financial implications are discussed in Section \ref{sec:numerical}.  Section \ref{sec:conslusion} summarizes the main results and discusses some future research directions. Finally, all technical proofs are consolidated in Section \ref{sec:proofs}.

{\footnotesize
\begin{center}
\begin{tikzpicture}[
    node distance=2.7cm,
    box/.style={draw, rectangle, minimum width=6cm, minimum height=2cm, align=center}
]
    \node[box] (def) {{\bf Optimal tracking problem \eqref{eq:w}}};
    \node[box, right=of def] (xt) {$D_t:=Z_t-V_t^{\theta,c}$\\ $L_t=0\vee \sup_{s\leq t}D_s$\\[0.4em] {\bf Auxiliary process}:~$X_t:=L_t-D_t$ \\{\bf Auxiliary problem \eqref{eq:u0}}};
    \node[box, below=of xt] (ft) {{\bf Dual function }~$\hat{v}(y,z,m)$\\[0.4em] Solving linear dual PDE \eqref{HJB.3}\\ by {\bf Proposition} \ref{prop:sol-v}};
    \node[box, left=of ft] (int) {{\bf Value function} $v(x,z,m)$\\  {\bf Optimal feedback strategy}\\  (see {\bf Theorem}  \ref{thm:verification})};

    \draw[-Latex] (def) -- (xt);
    \draw[-Latex] (xt) -- (ft)
     node[midway,right]{Legendre-Fenchel transform};
    \draw[-Latex] (ft) -- (int)
    node[midway,above] {inverse transform};
    \draw[-Latex] (int) -- (def);
\end{tikzpicture}
\end{center}}

\section{Problem Formulation and Equivalent Auxiliary Problem}\label{sec:model}
\subsection{Market model and problem formulation}

Let  $(\Omega, \mathcal{F}, \Fx,\mathbb{P})$ be a filtered probability space with the filtration $\mathbb{F}=(\mathcal{F}_t)_{t\geq 0}$ satisfying the usual conditions. Consider a financial market consisting of $d$ risky assets whose price dynamics follows the Black-Scholes model that
\begin{align}\label{stockSDE}
dS_t = {\rm diag}(S_t)(\mu dt+\sigma dW_t),\quad S_0\in(0,\infty)^d,\ \ t> 0,
\end{align}
where $S=(S_t^1,\ldots,S_t^d)_{t\geq0}^{\top}$ is the price process vector of $d$ risky assets,  and $W=(W_t^1,\ldots,W_t^d)_{t\geq 0}^{\top}$ is a $d$-dimensional $\Fx$-adapted Brownian motion. Moreover, $\mu=(\mu_1,\ldots,\mu_d)^{\top}\in\R^d$ denotes the vector of return rate and $\sigma=(\sigma_{ij})_{d\times d}$ is the volatility matrix that is assumed to be invertible. It is also assumed that the riskless interest rate $r=0$, which amounts to the change of num\'{e}raire. From this point onwards, all processes including the wealth process and the benchmark process are defined after the change of num\'{e}raire.

At time $t\geq 0$, let $\theta_t^i$ be the amount of wealth that the fund manager allocates in asset $S^i=(S_t^i)_{t\geq 0}$, and let $c_t$ be the non-negative consumption rate. The self-financing wealth process under the portfolio $\theta=(\theta_t^1,\ldots,\theta_t^d)_{t\geq 0}^{\top}$ and the consumption strategy $c=(c_t)_{t\geq 0}$ is given by
\begin{align}\label{eq:wealth2}
dV^{\theta,c}_t &=\theta_t^{\top}\mu dt+ \theta_t^{\top}\sigma dW_t-c_tdt,\quad t>0
\end{align}
with $V_0^{\theta,c}=\textrm{v}\geq 0$ being the initial wealth of the fund manager.

In the present paper, we consider the situation when the fund manager also concerns the relative performance with respect to an external benchmark process, which is described as the following geometric Brownian motion (GBM) that
\begin{align}\label{eq:Zt}
d Z_t=\mu_Z Z_t d t+\sigma_Z Z_t d W_t^{\gamma}, \quad Z_0=z\geq0,
\end{align}
where the return rate $\mu_Z\in\R$, the volatility $\sigma_Z\geq0$, and the Brownian motion $W_t^{\gamma}:=\gamma^{\top}W_t$ for $t\geq0$ and $\gamma=(\gamma_1,\ldots,\gamma_d)^{\top}\in\R^d$ satisfying $|\gamma|=1$, i.e., the Brownian motion $W^{\gamma}=(W_t^{\gamma})_{t\geq0}$ is a linear combination of $W$ with weights $\gamma$. Unless specified otherwise, $|\cdot|$ refers to the Euclidean norm of vectors.

Given the benchmark process $Z=(Z_t)_{t\geq0}$, we consider the relaxed benchmark tracking formulation in \cite{BHY23a,BHY23b} in the sense that the fund manager can strategically chooses the dynamic portfolio and consumption as well as the fictitious capital injection such that the total capital outperforms the benchmark process at all times. That is, the fund manager optimally chooses the regular control $\theta=(\theta_t)_{t\geq0}$ as the dynamic portfolio in risky assets, the regular control $c$ as the consumption rate and the singular control $A=(A_t)_{t\geq0}$ as the cumulative capital injection such that $A_t+V_t^{\theta,c}\geq Z_t$ for all $t\geq0$. Furthermore, consider a drawdown constraint on the consumption rate in the sense that $c_t$ cannot fall below a fraction $\lambda  \in [0, 1]$ of its past maximum that 
\begin{align}
c_t\geq \lambda M_t,\quad \forall t\geq 0,
\end{align}
where, the non-decreasing reference process $M=(M_t)_{t\geq 0}$ is defined as the historical spending maximum that
\begin{align}\label{eq:Mt}
M_t:=\max\left\{m,\sup_{s\in[0,t]} c_s\right\},\quad \forall t\geq0
\end{align}
with $m\geq 0$ being the initial reference level. The goal of the agent is to maximize the
expected utility on consumption deducted by the cost of capital injection in the sense that, for all $(\mathrm{v},z,m)\in\R_+^3$ with $\R_+:=[0,\infty)$,
\begin{align}\label{eq:w}
\begin{cases}
\displaystyle {\rm w}(\mathrm{v},z,m):=\sup_{(\theta,c,A)\in\mathbb{U}} \Ex\left[\int_0^{\infty} e^{-\rho t} U(c_t)dt- \beta\int_0^{\infty} e^{-\rho t}dA_t\right],\\[1.6em]
\displaystyle \text{subject to}~ A_t + V^{\theta,c}_t\ge Z_t~\text{for all}~ t\geq 0.
\end{cases}
\end{align}
 Here, we define the admissible control set $\mathbb{U}$  as the set of adapted processes $(\theta,c,A)=(\theta_t,c_t,A_t)_{t\geq0}$ such that
\begin{itemize}
    \item  $(\theta,c)$ are $\Fx$-adapted processes taking values on $\R^d\times\R_+$ satisfying the control drawdown constraint $c_t\ge \lambda M_t$,  and  the integrability condition~$\Ex[\int_0^t (c_s+|\theta_s|^2)ds]<\infty$  for $t\geq0$, 
    \item $A$ is a non-negative, non-decreasing and $\Fx$-adapted (r.c.l.l.) process with initial value $A_0=a\in\R_+$.
\end{itemize}
The constant $\rho>0$ is the subjective discount rate, and  the parameter $\beta>0$ denotes the utility per injected capital, which describes the relative importance between the consumption performance and the cost of capital injection. We consider the CRRA utility in the paper that
\begin{align}\label{eq:Ui}
U(x)=\frac{1}{p}x^{p},
\end{align}
where $1-p\in(0,1)\cup (1,+\infty)$ is the risk averse parameter.

To tackle problem \eqref{eq:w} with the floor constraint, we reformulate the problem based on the observation that, for a fixed control pair $(\theta,c)$, the optimal $A$ is always the smallest adapted right-continuous and non-decreasing process that dominates $Z-V^{\theta,c}$. It follows from Lemma 2.4 in \cite{BoLiaoYu21} that,  for fixed regular control pair $(\theta,c)$, the corresponding optimal singular control $A^{(\theta,c),*}=(A^{(\theta,c),*}_t)_{t\geq0}$ satisfies that $A^{(\theta,c),*}_t =0\vee \sup_{s\leq t}(Z_{s}-V_{s}^{\theta,c}),~\forall t\geq 0$. Thus, problem \eqref{eq:w} admits the equivalent formulation with a running maximum cost that
\begin{align}\label{eq_orig_pb}
{\rm w}(\mathrm{v}, z,m) &=-\beta(z-\mathrm{v})^+\nonumber\\
&\quad+\sup_{(\theta,c)\in\mathbb{U}^{\rm r}}\ \Ex\left[ \int_0^{\infty} e^{-\rho t} U(c_t)dt-\beta\int_0^{\infty} e^{-\rho t} d\left(0\vee \sup_{s\leq t}(Z_{s}-V_{s}^{\theta,c})\right)\right],
\end{align}
where $\mathbb{U}^{\rm r}$ is the admissible control set of pairs $(\theta,c)=(\theta_t,c_t)_{t\geq0}$ that will be specified later.

\begin{remark}
Note that, the floor state constraint $A_t+V_t^{\theta,c}\geq Z_t$ disappears in the formulation above, while we still need to cope with the control drawdown constraint $c_t\geq \lambda M_t$ for problem \eqref{eq_orig_pb}, which leads to a free boundary condition of the associated HJB equation that differs fundamentally from \cite{BHY23a,BHY23b}. In addition, we would like to stress that $A_t^{(\theta,c),*}=\sup_{s\leq t}(V_s^{\theta,c}-Z_s)^-$ under the control pair $(\theta,c)$ in fact records the largest shortfall when the wealth process $V_s^{\theta,c}$ falls below the benchmark process $Z_s$ up to time $t$. Consequently, when the strategic capital injection is not possible in the fund management, we can also directly consider the problem formulation \eqref{eq_orig_pb} to allow the wealth process $V_t^{\theta,c}$ to fall below $Z_t$ from time to time. However, we need to control the size of the expectation $\mathbb{E}[\int_0^{\infty}e^{-\rho t}d(0\vee \sup_{s\leq t}(Z_s-V_s^{\theta, c}))]$, which can be interpreted as the expected largest shortfall of the wealth with respect to the benchmark in a long run.

For the optimal consumption problem under a drawdown constraint, a lower bound on initial wealth is typically required to ensure the well-posedness (see, e.g., \citealt{Arun12, JeonPark21, JO22}). By considering an exit time, \cite{Bahman2019} avoids this assumption by terminating the problem when the wealth process reaches zero. In contrast, under the relaxed benchmark tracking formulation \eqref{eq:w}, we allow the wealth process $V^{\theta,c}=(V_t^{\theta,c})_{t\geq0}$ to take even negative values. Here, the feasibility of both the drawdown constraint and the benchmark constraint is guaranteed by the possible capital injection (or equivalently, by relaxing the bankruptcy constraint but controlling the expected largest shortfall with respect to the benchmark).
\end{remark}

\subsection{Equivalent auxiliary control problem}\label{sec:aux}

This subsection introduces a more tractable auxiliary stochastic control problem, which is mathematically equivalent to problem \eqref{eq_orig_pb}. To this end, we define a new controlled state process to replace the original state process $V^{\theta,c}=(V_t^{\theta,c})_{t\geq 0}$ in \eqref{eq:wealth2}. Then, consider the distance process $D_t:=Z_t-V_t^{\theta,c},~\forall t\geq 0$
with $D_0=z-\mathrm{v}$, and its running maximum process given by $L_t:=0\vee \sup_{s\leq t}D_s\geq 0$ for $t\geq0$, and $L_0=0$. Thus, the new controlled state process $X=(X_t)_{t\geq 0}$ taking values on $\R_+$ is defined as the reflected process $X_t:=L_t-D_t$ for $t\geq 0$ that satisfies the following SDE with reflection:
\begin{align}\label{state-X}
X_t =x+\int_0^t\theta_s^{\top}\mu ds+\int_0^t\theta_s^{\top}\sigma dW_s  -\int_0^t c_s ds-\int_0^t \mu_Z Z_sds-\int_0^t \sigma_Z Z_sdW_s^{\gamma}+ L_t
\end{align}
with the initial value $X_0=x:=(\mathrm{v}-z)^+\in\R_+$. For the notational convenience, we have omitted the dependence of $X=(X_t)_{t\geq0}$ on the control $(\theta,c)$. In particular, the process $L=(L_t)_{t\geq0}$ which is referred to as the local time of $X$, it increases at time $t$ if and only if $X_t=0$, i.e., $L_t=D_t$. We will change the notation from $L_t$ to $L_t^X$ from this point on wards to emphasize its dependence on the new state process $X$ given in \eqref{state-X}.

With the above preparations, consider the auxiliary stochastic control problem that, for $(x,z,m)\in\R_+^3$,
\begin{align}\label{eq:u0}
\begin{cases}
\displaystyle v(x,z,m):=\sup_{(\theta,c)\in\mathbb{U}^{\rm r}}J(x,z,m;\theta,c)\\
\displaystyle\qquad\qquad:=\sup_{(\theta,c)\in\mathbb{U}^{\rm r}}\Ex\left[\int_0^\infty e^{-\rho t} U(c_t)dt- \beta \int_0^{\infty} e^{-\rho t}dL_t^X\Big|X_0=x,Z_0=z,M_0=m \right],\\[1em]
\displaystyle \text{s.t.~the state process}~(X,Z,M)~\text{satisfies the dynamics~\eqref{state-X},~\eqref{eq:Zt}~and~\eqref{eq:Mt}},
\end{cases}
\end{align}
 where the admissible control set $\mathbb{U}^{\rm r}$ is specified as the set of $\Fx$-adapted processes $(\theta,c)=(\theta_t,c_t)_{t\geq0}$ taking values on $\R^d\times\R_+$ such that the drawdown constraint $c_t\geq \lambda M_t$ and the integrability condition~$\Ex[\int_0^t (c_s+|\theta_s|^2)ds]<\infty$ for $t\geq0$. It is not difficult to observe the following equivalence result.
\begin{lemma}\label{lem:equivalence}
For value functions $\mathrm{w}({\rm v},z,m)$ defined in \eqref{eq_orig_pb} and  $v(x,z,m)$ defined in \eqref{eq:u0}, we have
${\rm w}(\mathrm{v},z,m)=v((\mathrm{v}-z)^+,z,m)-\beta ( z-\mathrm{v})^+$ for all $(\mathrm{v},z,m)\in\R_+^3$.
\end{lemma}

It is straightforward to derive the following property of the value function $v$ in \eqref{eq:u0}.
\begin{lemma}\label{lem:propoertyu}
 The value function $x\to v(x,z,m)$ given by \eqref{eq:u0} is non-decreasing. Furthermore, for all $(x_1,x_2,z,m)\in\R_+^4$, we have
\begin{align}\label{eq:lipvxz}
\left|v(x_1,z,m)-v(x_2,z,m)\right|\leq \beta |x_1-x_2|.
\end{align}
\end{lemma}

Applying the dynamic programming arguments, the associated HJB variational inequality (HJB-VI) with the Neumann boundary condition can be written as, for $(x,z,m)\in\R_+^3$,
\begin{align}\label{eq:HJB-v}
\begin{cases}
\displaystyle\max\bigg\{\sup_{\theta\in\R^d}\left[\theta^{\top}\mu v_x+\frac{1}{2}\theta^{\top}\sigma\sigma^{\top}\theta v_{xx}+ \theta^{\top} \sigma \gamma\sigma_Z z (v_{x z}-v_{xx})  \right]+\sup_{c\in[\lambda m,m]}\left(U(c)-cv_x\right)\\[1.2em]
\displaystyle\qquad\qquad-\sigma_Z^2 z^2 v_{xz}+\frac{1}{2}\sigma_Z^2 z^2 (v_{xx}+v_{zz})+\mu_Z z (v_z-v_x)-\rho v,v_m\bigg\}=0,\\[1em]
\displaystyle v_x(0,z,m)=\beta,\quad \forall (z,m)\in\R^2_+,
\end{cases}
\end{align}
 where, the Neumann boundary condition $ v_x(0,z,m)=\beta$ stems from the martingale optimality condition because the local time process $L^X_t$ increases whenever the process $X_t$ hits $0$.

If one assumes heuristically that $v_{xx}<0$ and  $v_x>0$, which will be verified later, the feedback optimal control can be uniquely determined by, for $(x,z,m)\in\R_+^3$,
\begin{align}\label{feedback.opti}
\begin{cases}
\displaystyle \theta^{*}(x,z,m)=-(\sigma\sigma^{\top})^{-1}\frac{v_x(x,z,m)\mu+(v_{xz}-v_{xx})(x,z,m)z\sigma_Z\sigma\gamma}{v_{xx}(x,z,m)},\\[1em]
\displaystyle c^*(x,z,m)=\max\{\lambda m,\min\{m,v^{\frac{1}{p-1}}_x(x,z,m)\}\}.
\end{cases}
\end{align}
Plugging \eqref{feedback.opti} into \eqref{eq:HJB-v}, we get
\begin{align}\label{HJB.2}
\begin{cases}
\displaystyle\max\Bigg\{ -\alpha\frac{v^2_x}{v_{xx}}+\frac{1}{2}\sigma^2_Zz^2\left(v_{zz}-\frac{v^2_{xz}}{v_{xx}}\right)-\eta z\frac{v_xv_{xz}}{v_{xx}}+(\eta-\mu_Z)zv_x+\mu_Zzv_z \\
\displaystyle\qquad\quad+\frac{1}{p}(c^*)^{p}-c^*v_x-\rho v,v_m\Bigg\}=0,\\
\displaystyle v_x(0,z,m)=\beta
\end{cases}
\end{align}
with the coefficients $\alpha:=\frac{1}{2}\mu^{\top}(\sigma\sigma^{\top})^{-1}\mu$ and $\eta:=\sigma_Z\gamma^{\top}\sigma^{-1}\mu$. The next section will explore the solvability of HJB-VI \eqref{HJB.2} for the auxiliary control problem \eqref{eq:u0}. 

\section{Solvability of HJB-VI}\label{sec:solvHJBVI}

\subsection{Piecewise HJB-VI across different regions}\label{sec:3-1}

We first heuristically decompose the domain $\R_+^3$ into the following five regions such that Eq. \eqref{HJB.2} can be expressed piecewisely depending on the optimal consumption control, where we set $y_1(m)> y_2(m)\geq y^*(m)$ for $m\geq0$ such that 
$y_1(m):=(\lambda m)^{p-1}$, $y_2(m):=m^{p-1}$ and $y^*(m)$ is the free boundary that will be determined later.  As a direct result of Lemma~\ref{lem:propoertyu}, if the value function $v$ is $C^1$ in $x$, then $|v_x(x,z,m)|= v_x(x,z,m)\leq \beta$ for all $(x,z,m)\in\R_+^3$, as it is assumed that $v_x>0$.  This implies that  $v^{\frac{1}{p-1}}_x(x,z,m)\geq \beta^{\frac{1}{p-1}}$ for all $(x,z,m)\in\R_+^3$, which provides a subsistent consumption constraint  arising from arise from the tracking formulation
by allowing capital injection. Thus, let us introduce the domain ${\cal O}:=\{(x,z,m)\in\R^3;~m\geq \beta^{\frac{1}{p-1}}\}$.

{\it Region I:} On the set $\mathcal{R}_1=\{(x,z,m)\in{\cal O};~ y_1(m)< v_x(x,z,m)\leq\beta\}$, the optimal consumption $c^*(x,z,m)=\lambda m$. In this case, the wealth level is very low such that it is optimal for the fund manager to consume at the drawdown constraint level. 


{\it Region II:} On the set $\mathcal{R}_2=\{(x,z,m)\in{\cal O};~y_2(m)< v_x(x,z,m)\leq y_1(m)\}$, the optimal consumption $c^*(x,z,m)=v_x(x,z,m)^{\frac{1}{p-1}}$. In this case, the wealth level is at an intermediate level such that the consumption rate is greater than the lowest rate and lower than its historical peak. 


{\it Region III:} On the set $\mathcal{R}_3=\{(x,z,m)\in{\cal O};~y^*(m)< v_x(x,z,m)\leq y_2(m)\}$, the optimal consumption $c^*(x,z,m)=m$. In this case, the wealth level is large enough that the optimal consumption rate is either to revisit its historical peak from below or to sit on the same peak. 


{\it Region IV:} On the set $\mathcal{R}_4=\{(x,z,m)\in{\cal O};~v_x(x,z,m)= y^*(m)\}$, the optimal consumption $c^*(x,z,m)=m$. In this case, the wealth level is large enough such that the optimal consumption rate $c^*(x,z,m)$ which is a singular control creates a new record of the peak and $M_t=c_t^*$ is strictly increasing at the instant time. Thus, we have to mandate the free boundary condition $v_m(x,z,m)=0$ and a so-called ``super-contact condition" $v_{mx}(x,z,m)=v_{mz}(x,z,m)=0$. 

{\it Region V:} On the set $\mathcal{R}_5=\{(x,z,m)\in\R^3_+;~(x,z,m)\in{\cal O}^c~\text{or}~v_x(x,z,m)< y^*(m)\}$, the optimal consumption strategy $c^*(x,z,m)>m$, which indicates that the initial level $m$ is below the feedback control $c^*(x,z,m)$ and the historical peak $M=(M_t)_{t\geq0}$ has a jump immediately to attain $c^*(x,z,m)$. As a result, for any initial value $(x,z,m)$ in the set $\mathcal{R}_5$, the feedback control $c^*(x,z,m)$ will push the current states jumping immediately to the point $(x,z,m^*)$ on the region $\mathcal{R}_4$ with $m^*=c^*(x,z,m)$.

In sum, it is sufficient to only consider $(x,z,m)$ on the effective domain ${\cal D}:=\bigcup_{i=1}^4{\cal R}_i$. As a result, for $(x,z,m)\in\mathcal{D}$, the HJB-VI \eqref{HJB.2} can be rewritten in the piecewise form that  
\begin{align}\label{eq:HJB-VI99}
\begin{cases}
 \displaystyle -\alpha\frac{v^2_x}{v_{xx}}+\frac{1}{2}\sigma^2_Zz^2\left(v_{zz}-\frac{v^2_{xz}}{v_{xx}}\right)-\eta z\frac{v_xv_{xz}}{v_{xx}}+(\eta-\mu_Z)zv_x+\mu_Zzv_z+\Phi(v_x)-\rho v=0,\\[1.2em]
 \displaystyle v_m(x,z,m)\leq0,~~\displaystyle v_x(0,z,m)=\beta,\\[0.8em]
\displaystyle v_m(x,z,m)=v_{mz}(x,z,m)=v_{mx}(x,z,m)=0 \text{ if }v_{x}(x,z,m)=y^*(m),
\end{cases}
\end{align}
where, the piecewise mapping $\Phi:(0,\beta]\mapsto\R$ is defined by, for all $y\in(0,\beta]$,
\begin{align}\label{eq:Phi}
\Phi(y):=\begin{cases}
    \displaystyle \frac{1}{p}(\lambda m)^p-\lambda my, & y\geq (\lambda m)^{p-1}, \\[1.2em]
    \displaystyle \frac{1-p}{p}y^{\frac{p}{p-1}},&m^{p-1}<y<(\lambda m)^{p-1},\\[0.8em]
    \displaystyle~  \frac{1}{p}m^p-my,&y\leq m^{p-1}.
\end{cases}
\end{align}
Then, we may apply Legendre-Fenchel transform of the solution $v$ only with respect to $x$ that, for all $(y,z,m)\in[y^*(m),\beta]\times\mathbb{R}_+^2$,
\begin{align}\label{eq:LFT-v}
\hat{v}(y,z,m):=\sup_{x>0}\{v(x,z,m)-xy\}. 
\end{align}
Hence, $v(x,z,m)=\inf_{y\in(0,\beta]}(\hat{v}(y,z,m)+xy)$ for $(x,z,m)\in\mathcal{D}$. Define $x^*(y,z,m)=v_x(\cdot,z,m)^{-1}(y)$ with $y\to v_x(\cdot,z,m)^{-1}(y)$ being the inverse function of $x\to v_x(x,z,m)$. Thus, $x^*=x^*(y,z,m)$ satisfies the equation $v_x(x^*,z,m)=y$ for $(z,m)\in\R_+^2$. Using the relationship between $y$ and $v_x(x^*,y,m)$, we get the 
dual PDE of HJB-VI \eqref{HJB.2} with both Neumann/free boundary conditions that, for $(y,z,m)\in [y^*(m),\beta]\times\R_+^2$,
\begin{align}\label{HJB.3}
\begin{cases}
\displaystyle -\rho\hat{v}+\rho y\hat{v}_y+\alpha y^2\hat{v}_{yy}
     +\mu_Zz\hat{v}_z+\frac{\sigma^2_Z}{2}z^2\hat{v}_{zz}-\eta zy\hat{v}_{yz}-(\mu_Z-\eta)zy
     +\Phi(y)=0,\\[0.8em]
\displaystyle \hat{v}_y(\beta,z,m)=0,\\[0.8em]
\displaystyle \hat{v}_m(y^*(m),z,m)=\hat{v}_{ym}(y^*(m),z,m)=\hat{v}_{zm}(y^*(m),z,m)=0
\end{cases}
\end{align}
with the gradient constraint $\hat{v}_m(y,z,m)\leq0$.

\subsection{Derivation of the solution to the dual PDE}\label{sec:solu}

The next result gives the closed-form solution of the dual PDE \eqref{HJB.3}. 
\begin{proposition}\label{prop:sol-v}
Let $\mu_Z\geq \eta$. Consider the piecewise mapping $y^*:[\beta^{\frac{1}{p-1}},\infty)\mapsto(0,\infty)$ defined by, for $m\in[\beta^{\frac{1}{p-1}},\frac{1}{\lambda}\beta^{\frac{1}{p-1}})$, $y^*(m):= m^{p-1}$; while for $ m\geq \frac{1}{\lambda}\beta^{\frac{1}{p-1}}$, $y^*(m)$ is the unique solution to the equation:
\begin{align}\label{eq:r-1}
&\frac{\beta m^{p-1}}{(\alpha+\rho)y^*(m)}+\frac{\beta}{\alpha+\rho}\ln\left(\frac{y^*(m)}{\beta}\right)+\frac{\rho}{(\alpha+\rho)^2}\beta^{-\frac{\rho}{\alpha}} (y^*(m))^{\frac{\alpha+\rho}{\alpha}}-\frac{m^{p-1}}{\alpha+\rho} \beta^{-\frac{\rho}{\alpha}}(y^*(m))^{\frac{\rho}{\alpha}}\nonumber\\
&\quad=
\frac{\alpha\beta^{-\frac{\rho}{\alpha}}}{(\alpha+\rho)^2}\left(\lambda^{\frac{\alpha p-(1-p)\rho}{\alpha}}-1\right)m^{-\frac{(\alpha+\rho)(1-p)}{\alpha}}+\frac{\beta\lambda}{\alpha+\rho}\ln\beta(\lambda m)^{1-p}-\frac{\beta}{\alpha+\rho}\ln(\beta m^{1-p})\nonumber\\
&\qquad+\frac{\lambda \beta}{\alpha+\rho}-\frac{\alpha\beta}{(\alpha+\rho)^2}+\frac{(1-p)^2\beta(\lambda-1)}{p(\alpha+\rho)}-\frac{\beta(\alpha+\rho+p\alpha)(\lambda-1)}{p(\alpha+\rho)^2}+\frac{\beta(1-p)(\lambda-1)}{\alpha+\rho}.
\end{align}
Then, we have:
\begin{itemize}
\item[{\rm(i)}] the mapping $m\mapsto y^*(m)$ is well-defined. Furthermore, $m\to y^*(m)$ is strictly decreasing and satisfies $\lim_{m\to\infty}y^*(m)=0$;
\item[{\rm(ii)}] the solution of the dual PDE \eqref{HJB.3} admits the following closed-form given by
{\small \begin{align}\label{sol:dual-v}
\hat{v}(y,z,m)=
\begin{cases}
 \displaystyle   \frac{1}{\beta}C_1(m)y+\beta^{\frac{\rho}{\alpha}} C_2(m)y^{-\frac{\rho}{\alpha}}+\frac{(\lambda m)^p}{p\rho}+\frac{\lambda m}{\alpha+\rho}y\ln\left(\frac{y}{\beta}\right)\\[0.8em]
 \displaystyle \hspace{4cm} +z\left( y-\frac{\beta^{-\kappa+1}}{\kappa}y^{\kappa}\right),\qquad (\lambda m)^{p-1}<y\leq \beta,  \\[1em]
  \displaystyle    \frac{1}{\beta}C_3(m)y+\beta^{\frac{\rho}{\alpha}}C_4(m)y^{-\frac{\rho}{\alpha}}+\frac{(1-p)^3}{p(\rho(1-p)-\alpha p)}y^{\frac{p}{p-1}}\\[0.8em]
  \displaystyle  \hspace{4cm} +z\left( y-\frac{\beta^{-\kappa+1}}{\kappa}y^{\kappa}\right),\qquad  m^{p-1}<y\leq (\lambda m)^{p-1},\\[1em]
  \displaystyle    \frac{1}{\beta}C_5(m)y+\beta^{\frac{\rho}{\alpha}}C_6(m)y^{-\frac{\rho}{\alpha}}+\frac{m^p}{p\rho}+\frac{ m}{\alpha+\rho}y\ln\left(\frac{y}{\beta}\right)\\[0.7em]
  \displaystyle \hspace{4cm} +z\left( y-\frac{\beta^{-\kappa+1}}{\kappa}y^{\kappa}\right),\qquad y^*(m)\leq  y\leq m^{p-1},
\end{cases}
\end{align}}where, the constant $\kappa$ is given by 
$\kappa:=\frac{-(\rho-\eta-\alpha)+\sqrt{(\rho-\eta-\alpha)^2+4\alpha(\rho-\mu_Z)}}{2\alpha}>0$,
and the coefficients $m\mapsto C_i(m)$ for $i=1,\ldots,6$ are given by \eqref{eq:C1}-\eqref{eq:C6} in Section \ref{sec:proofs}.
\end{itemize}
\end{proposition}

Let us introduce the constant 
\begin{align}\label{boundrho0}
    \rho_0:=
    \begin{cases}
        \max\{\mu_Z,\max\{2\alpha,\alpha p/(1-p)\}\},&\text{if}~p\in (0,1),\\[0.5em]
        \max\{\mu_Z,0\},&\text{if}~p<0.
    \end{cases}
    \end{align}
Based on Proposition~\ref{prop:sol-v}, we further have the next result.
\begin{lemma}\label{lem:dual-sol}
Let $\mu_Z\geq \eta$  and $\rho>\rho_0$. Then, the function $y\to \hat{v}(y,z,m)$ defined in Proposition \ref{prop:sol-v} is continuous, strictly convex and decreasing.
\end{lemma}

\section{Main Results}\label{sec:verification}

In this section, we will show that the value function of problem \eqref{eq:u0} is the inverse transform of $\hat{v}(y,z,m)$ given in \eqref{sol:dual-v} such that we can characterize the optimal investment and consumption in feedback form in terms of the primal variables.

We start with the characterization of the inverse transform of $\hat{v}(y,z,m)$ in \eqref{sol:dual-v}. To do it, introduce the following functions defined on  $(z,m)\in\R_+\times [\beta^{\frac{1}{p-1}},+\infty)$ given by
\begin{align*}
F_1(z,m):=\begin{cases}
\displaystyle -\hat{v}_{y}((\lambda m)^{p-1},z,m)=
-\frac{1}{\beta}C_1(m)+\frac{\rho}{\alpha}\beta^{\frac{\rho}{\alpha}}C_2(m)(\lambda m)^{\frac{(\alpha+\rho)(1-p)}{\alpha}},\\
\displaystyle \qquad\quad-\frac{\lambda m}{\alpha+\rho}\left(\ln\beta^{-1}(\lambda m)^{p-1}+1\right)-z\left(1-\beta^{1-\kappa}(\lambda m)^{(1-\kappa)(1-p)}\right), &\text{if} ~m\geq\frac{1}{\lambda} \beta^{\frac{1}{p-1}}\\
\displaystyle\quad 0, &\text{if}~ m\leq\frac{1}{\lambda} \beta^{\frac{1}{p-1}},
\end{cases}
\end{align*}
and
\begin{align*}
F_2(z,m):=-\hat{v}_{y}( m^{p-1},z,m)&=
-\frac{1}{\beta}C_3(m)+\frac{\rho}{\alpha}\beta^{\frac{\rho}{\alpha}}C_4(m)m^{\frac{(\alpha+\rho)(1-p)}{\alpha}}
\nonumber\\
&\quad+\frac{(1-p)^2}{\rho(1-p)-\alpha p}m-z\left(1-\beta^{1-\kappa}m^{(1-\kappa)(1-p)}\right),\nonumber\\
F_3(z,m):=-\hat{v}_{y}(y^*(m),z,m)&=
-\frac{1}{\beta}C_5(m)+\frac{\rho}{\alpha}\beta^{\frac{\rho}{\alpha}}C_6(m)(y^*(m))^{-\frac{\alpha+\rho}{\alpha}}\nonumber\\
&\quad
-\frac{m}{\alpha+\rho}(\ln\beta^{-1}y^*(m)+1)-z\left(1-\beta^{1-\kappa}(y^*(m))^{\kappa-1}\right),
\end{align*}
where, the function $\hat{v}(y,z,m)$ is given by \eqref{sol:dual-v}. Then, it follows from Lemma \ref{lem:dual-sol} that $0\leq F_1(z,m)\leq F_2(z,m)\leq F_3(z,m)$ for all $(z,m)\in\R_+\times [\beta^{\frac{1}{p-1}},+\infty)$.

\begin{lemma}\label{lem:boundary-m}
Let $\mu_Z\geq \eta$  and $\rho>\rho_0$. For fixed $z\in\R_+$, let $x\mapsto m^*(x,z)$ be the inverse function of $m\mapsto F_3(z,m)$.  Then, the function $m^*(x,z)$ with $(x,z)\in\R_+^2$ is well-defined. Moreover, for the parameter $\lambda\in(0,1]$, there exists a constant $C>0$ such that
\begin{align*}
m^*(x,z)\ln(\beta (m^*(x,z))^{1-p})\leq C(1+ x), \quad \forall (x,z)\in\R_+^2.
\end{align*}
\end{lemma}

Consider the inverse transform of \eqref{sol:dual-v} that, for all $ (x,z,m)\in\mathcal{D}:=\{(x,z,m)\in{\cal O};~x\leq F_3(z,m)\}$,
\begin{align}\label{v.func.def1}
v(x,z,m)=\inf_{y\in(0,\beta]}\{\hat{v}(y,z,m)+yx\}.
\end{align}
Furthermore, for the inverse function $x\mapsto m^*(x,z)$ of $m\mapsto F_3(z,m)$ given in Lemma \ref{lem:boundary-m}, define
\begin{align}\label{v.func.def2}
 v(x,z,m)=v(x,z,m^*(x,z)),\quad\forall (x,z,m)\in \R^3_+\backslash\mathcal{D}.
\end{align}
The following lemma characterizes the function  $v(x,z,m)$ with $(x,z,m)\in\R_+^3$ defined by \eqref{v.func.def1}-\eqref{v.func.def2}.
\begin{lemma}\label{lem:valuefunc-v}
Let $\mu_Z\geq \eta$  and $\rho>\rho_0$.  Then, $v(x,z,m)$ for $(x,z,m)\in\R_+^3$ is well-defined and  $v \in C^2(\R_+^3)$. Moreover, on the region ${\cal D}$, $v(x,z,m)$ satisfies the equation given by
\begin{align}\label{eq:equation-v1}
\sup_{\theta\in\R^d}\left[\theta^{\top}\mu v_x+\frac{1}{2}\theta^{\top}\sigma\sigma^{\top}\theta v_{xx}+ \theta^{\top} \sigma \gamma\sigma_Z z (v_{x z}-v_{xx})  \right]+\sup_{c\in[\lambda m,m]}\left(U(c)-cv_x\right)\nonumber\\
\qquad\qquad-\sigma_Z^2 z^2 v_{xz}+\frac{1}{2}\sigma_Z^2 z^2 (v_{xx}+v_{zz})+\mu_Z z (v_z-v_x)-\rho v=0
\end{align}
with Neumann boundary condition $ v_x(0,z,m)=\beta$  and {the free boundary condition that} 
$$v_m(x,z,m^*(x,z))=v_{xm}(x,z,m^*(x,z))=v_{zm}(x,z,m^*(x,z))=0.$$ 
On the region $\R^3_+\backslash\mathcal{D}$,  $v(x,z,m)$ satisfies Eq.~\eqref{eq:equation-v1} with  $v_m(x,z,m)=0$ and the boundary condition $v_x(0,z,m)=\beta$. 
\end{lemma}

\begin{figure}[h]
    \centering
\includegraphics[width=8cm]{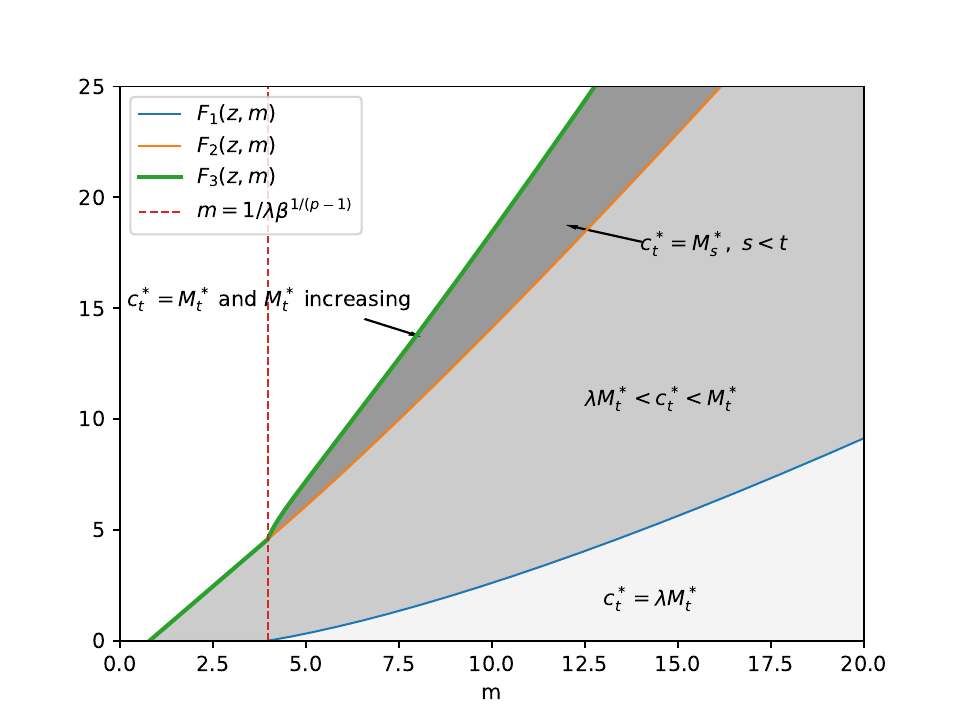}
\caption{{\small Boundary curves in terms of $m$ with parameters $\rho=2,\,p=-2,\,\lambda=0.2,\,\mu=0.01,\,\sigma=0.02,\,\beta=2,\,\mu_Z=\sigma_Z=0.5$ and $z=20$.}}\label{F_m}
\end{figure}

Next, we provide the optimal feedback strategy of portfolio and consumption in terms of $(x,z,m)\in\R_+^3$. We first have 
\begin{lemma}\label{lem:feedback-control}
Assume that $\mu_Z\geq \eta$  and $\rho>\rho_0$. Let $x\mapsto f(x,z,m)$ be the inverse function of $y\mapsto-\hat{v}_y(y,z,m)$. Introduce the (feedback) control functions as follows, for $(x,z,m)\in\mathcal{D}$,
\begin{align}\label{theta.def}
&\theta^*(x,z,m):=-(\sigma\sigma^{\top})^{-1}\frac{v_x(x,z,m)\mu+(v_{xz}-v_{xx})(x,z,m)z\sigma_Z\sigma\gamma}{v_{xx}(x,z,m)},
\end{align}
and
\begin{align}\label{consum.def}
  c^*(x,z,m):=
  \begin{cases}
      \displaystyle ~~~~~\lambda m, & ~~~~0\leq x<F_1(z,m),\\[0.7em]
      \displaystyle (f(x,z,m))^{\frac{1}{p-1}}, & F_1(z,m)\leq x<F_2(z,m),\\[0.7em]
      \displaystyle ~~~~~~m, & F_2(z,m)\leq x\leq F_3(z,m).
  \end{cases}
\end{align}
For $(x,z,m)\in\R^3_+\backslash\mathcal{D}$, introduce that $\theta^*(x,z,m):=\theta^*(x,z,m^*(x,z))$ and $c^*(x,z,m):=c^*(x,z,m^*(x,z))=m^*(x,z)$.
Then, there exist positive constants $(M_{\theta}, M_c)$ such that, for all $(x,z,m)\in\R_+^3$,
\begin{align*}
|\theta^*(x,z,m)|\leq M_{\theta}(1+x+z),\quad |c^*(x,z,m)|\leq M_c(1+x+m).
\end{align*}
\end{lemma}

Now, we are ready to show the verification result, which proves that the function $v(x,z,m)$ introduced by \eqref{v.func.def1}-\eqref{v.func.def2} is indeed the value function for problem \eqref{eq:u0} and the admissible strategy induced by the feedback control functions $\theta^*(x,z,m)$ and $c^*(x,z,m)$ defined by \eqref{theta.def} is the optimal strategy of investment and consumption.

\begin{theorem}\label{thm:verification}
Let $\mu_Z\geq \eta$ and $\rho>\rho_0$ with $\rho_0$ given by \eqref{boundrho0}. Recall the function $v(x,z,m)$ introduced by \eqref{v.func.def1}-\eqref{v.func.def2} and the feedback control function $(\theta^*(x,z,m),c^*(x,z,m))$ given by \eqref{theta.def}. 
Consider the controlled state process $(X^*,Z,M^*)=(X^*_t,Z_t,M^*_t)_{t\geq0}$ that obeys the following reflected SDE, for $(t,x,z,m)\in\R_+^4$,
\begin{align}\label{eq:optimal-SDE}
\begin{cases}
\displaystyle X^*_t=x+\int_0^t(\theta^*(X_s^*,Z_s,M_s^*))^{\top}\mu ds+\int_0^t(\theta^*(X_s^*,Z_s,M_s^*))^{\top}\sigma dW_s \\[0.7em]
\displaystyle \qquad\quad -\int_0^t c^*(X_s^*,Z_s,M_s^*) ds-\int_0^t \mu_Z Z_sds-\int_0^t \sigma_Z Z_sdW_s^{\gamma}+ L^{X^*}_t,\\[0.8em]
\displaystyle M^*_t=\max\left\{m,\sup_{s\in[0,t]}m^*(X_s^*,Z_s)\right\},\\[1em]
\displaystyle Z_t=z+\int_0^{t}\mu_zZ_sds+\int_0^t\sigma_ZZ_sdW^{\gamma}_s
\end{cases} 
\end{align}
with $L^{X^*}_0=0$ and $M_0^*=m$. Define  $\theta^*_t=\theta^*(X^*_t,Z_t,M^*_t)$ and $c^*_t=c^*(X^*_t,Z_t,M^*_t)$ for all $t\geq0$. Then, the strategy pair $(\theta^*,c^*)=(\theta^*_t,c^*_t)_{t\geq0}\in\mathbb{U}^r$ is an optimal investment-consumption strategy in the sense that, for all admissible $(\theta,c)\in\mathbb{U}^r$,
\begin{align}\label{eq:value-func}
    \mathbb{E}\left[\int_0^{\infty}e^{-\rho t}U(c_t)dt-\beta\int_0^{\infty}e^{-\rho t}dL_t^X\right]\leq v(x,z,m),\quad \forall (x,z,m)\in\R^3_+,
\end{align}
and the equality holds when $(\theta,c)=(\theta^*,c^*)$.
\end{theorem}

\begin{remark}\label{rem:HJB-VI}
 Recall that Proposition \ref{prop:sol-v} provides an explicit classical solution  to the dual PDE \eqref{HJB.3} but not verify if the inequality constraint $\hat{v}_m(y,z,m)\leq 0$ is satisfied. It requires quite  tedious calculations due to the implicit expression of $y^*(m)$ and the coefficient functions $C_i(m)$ for $i=1,...,6$.
 In turn, Lemma \ref{lem:valuefunc-v} does not show that the value function defined by \eqref{v.func.def1}-\eqref{v.func.def2} satisfies the HJB-VI \eqref{eq:HJB-v} as it remains to prove $v_m(x,z,m)\leq 0$ for all $(x,z,m)\in\R_+^3$. However, it becomes an obvious result after Theorem \ref{thm:verification} in view of definition of admissible set $\mathbb{U}^r$ and the value function given by \eqref{eq:u0}.
\end{remark}

\begin{remark}\label{rem:XV}
In fact, the state processes of the primal control problem \eqref{eq:w} and the auxiliary control problem \eqref{eq:u0} satisfy the following relationship:
\begin{align}\label{eq:state}
X_t&=V_t^{\theta,c}-Z_t+\sup_{s\leq t}\left(Z_s-V_s^{\theta,c}\right)^+.
\end{align}
Therefore, we can obtain the auxiliary state process $(X_t)_{t\geq 0}$ by using the process $(V^{\theta,c},Z)=(V_t^{\theta,c},Z_t)_{t\geq 0}$. It shows that the optimal control $(\theta^*,c^*)=(\theta_t^*,c_t^*)_{t\geq0}$ in Theorem \ref{thm:verification} actually has the path-dependent structure in terms of the wealth process $V^{\theta,c}$ and the benchmark process $Z$, which will make the decision making intractable based on the direct study of the control problem using the original wealth process.  This justifies the main advantage of working with the auxiliary state process $X$ in the present paper, which significantly simplifies the problem and enables us to derive the optimal control $(\theta^*,c^*)$ in feedback form of $X$. 
\end{remark}

The following lemma shows that the expectation of the total optimal discounted capital injection is always finite and positive.
\begin{lemma}\label{lem:inject-captial}
Let the assumptions of Theorem \ref{thm:verification} hold. Consider the optimal strategy of investment and consumption $(\theta^*,c^*)=(\theta_t^*,c_t^*)_{t\geq 0}$ given in Theorem \ref{thm:verification}. Then, it holds that
\begin{itemize}
\item[{\rm(i)}] The  expectation of the discounted total capital injection under the optimal strategy $(\theta^*,c^*)$ is finite that, for all $(x,z,m)\in\R_+^3$,
\begin{align}\label{eq:dAstarfinite}
\Ex\left[\int_0^{\infty} e^{-\rho t}dA^*_t \right]\leq\frac{1}{|p|}\Ex\left[\int_0^{\infty} e^{-\rho t} Y_t^{\frac{p}{p-1}} dt\right]-v(x,z,m)
<+\infty.
\end{align}
\item[{\rm(ii)}] The  expectation of the discounted total capital injection under the optimal strategy $(\theta^*,c^*)$ is strictly positive that, for all $(x,z,m)\in\R_+\times (0,\infty)^2$,
\begin{align}\label{eq:dAstarpositive}
\Ex\left[\int_0^{\infty} e^{-\rho t}dA^*_t \right]\geq \max\left\{z\frac{1-\kappa}{\kappa}\left(1+\frac{x}{z}\right)^\frac{\kappa}{\kappa-1},\frac{\lambda m}{\alpha+\rho}e^{-\frac{\alpha+\rho}{\lambda m}x}\right\}>0.
\end{align}
\end{itemize}
 Here,  the optimal capital injection under the optimal strategy $(\theta^*,c^*)$ is given by $A^*_t =0\vee \sup_{s\leq t}(Z_{s}-V_{s}^{\theta^*,c^*})$ for $t\geq0$.
\end{lemma}

\begin{remark}
We impose two assumptions $\mu_Z\geq \eta$ and $\rho>\rho_0$ on the model parameters in this paper. From a mathematical standpoint, the condition $\mu_Z\geq \eta$  is not only important to show the strict convexity of the dual function $y \to \hat{v}(y,z,m)$, but is also crucial in demonstrating the strict concavity of the value function $x \to v(x,z,m)$ (cf. Lemma \ref{lem:dual-sol}), thereby it is needed for the well-posedness of the problem. Furthermore, the assumption $\rho>\rho_0$ is important in showing that  both  the value function and the capital injection are finite (cf. Theorem \ref{thm:verification} and  Lemma \ref{lem:inject-captial}). 

From an economic perspective, the assumption $\mu_Z \geq \eta$ implies that the fund manager in our problem only focuses on benchmark processes that perform sufficiently well in the market:
\begin{itemize}
    \item Case $1$ $(\sigma_Z = 0)$: the benchmark process reduces to $Z_t = z e^{\mu_Z t}$ for $t\geq0$, where $\mu_Z \geq 0$ describes a deterministic growth rate, as studied in \cite{YaoZZ06}.
    \item Case $2$ $(\sigma_Z > 0)$: the condition $\mu_Z \geq \eta$ is equivalent to $\frac{\mu_Z}{\sigma_Z} \geq \gamma^\top \sigma^{-1} \mu$, meaning the benchmark’s Sharpe ratio must be sufficiently high. In practice, fund managers would only select benchmarks with high Sharpe ratios, making this assumption not restrictive.
\end{itemize}
\end{remark}

The following result shows that, when the drawdown constraint vanishes as the parameter $\lambda=0$, problem \eqref{eq:u0} simplifies to the optimal tracking problem with no consumption constraint in \cite{BHY23a}.
\begin{corollary}\label{coro:no-constraint}
Let assumptions of Theorem \ref{thm:verification} hold.  Then, for fixed $(x,z,m)\in\R_+^3$, the value function $v(x,z,m)$ given by \eqref{eq:u0}.  is non-increasing w.r.t. the fraction parameter $\lambda$. In particular, when $\lambda=0$, the value function admits the form given by, for all $(x,z,m)\in\R_+^3$,
\begin{align}\label{valuefunc-non-constraint}
v(x,z,m)&=\frac{(1-p)^2}{\rho(1-p)-\alpha p}\beta^{\frac{1}{p-1}}f(x,z)+\frac{(1-p)^3}{p(\rho(1-p)-\alpha p)}f(x,z)^{\frac{p}{p-1}}+xf(x,z)\nonumber\\
&\quad+z\left( f(x,z)-\frac{\beta^{-\kappa+1}}{\kappa}f(x,z)^{\kappa}\right),
\end{align}
where the function $f(x,z)$ is uniquely determined by
\begin{align}\label{f-non-constraint}
x=\frac{(1-p)^2}{\rho(1-p)-\alpha p}\left(f(x,z)^{\frac{1}{p-1}}-\beta^{\frac{1}{p-1}}\right)+z\left( \beta^{-\kappa+1}f(x,z)^{\kappa-1}-1\right).
\end{align}
Furthermore, the optimal feedback control function is given by, for all $(x,z,m)\in\R_+^3$,
\begin{align}\label{feedbackcontrol-non-constraint}
\begin{cases}
\displaystyle \theta^*(x,z,m)=
 (\sigma\sigma^{\top})^{-1}\mu \left(\frac{1-p}{\rho(1-p)-\alpha p}f(x,z)^{\frac{1}{p-1}}+(1-\kappa) \beta^{-(\kappa-1)}zf(x,z)^{\kappa-1}\right)\\[0.8em]
\displaystyle \qquad\qquad\qquad+(\sigma\sigma^{\top})^{-1}\sigma_Z\sigma\gamma z \beta^{-(\kappa-1)}f(x,z)^{\kappa-1},\\[0.8em]
\displaystyle c^*(x,z,m)= f(x,z)^{\frac{1}{p-1}}.
\end{cases}
\end{align}
\end{corollary}

\section{Numerical Examples}\label{sec:numerical}
In this section, we present some numerical examples to illustrate the sensitivity analysis with respect to some model parameters and discuss their financial implications based on the optimal feedback functions and the expected total capital injection obtained in \eqref{theta.def}. To ease the discussions, we only consider the case $d=1$ in all examples.

Figure \ref{simulation} plots the simulated sample paths of the benchmark process $Z^*_t$, the controlled wealth process $V^*_t$, the cumulative capital injection $A^*_t$, the optimal consumption $c^*_t$ and the optimal consumption running maximum process $M^*_t$, respectively. We fix the model parameters $\rho=2,\,p=-0.1,\,\lambda=0.2,\,\mu=0.1,\,\sigma=0.1,\,\beta=2,\,\mu_Z=0.01,\,\sigma_Z=0.05,\,z=10,\,v=20,\,m=6$.
Initially, with $M^*_0=6$, $V^*_0=20$, and $Z_0=10$, higher wealth leads to a larger $c^*$ and growing $M^*$. As wealth fluctuates downwards, $c^*_t$ dynamically decreases to prevent costly capital injections. When wealth becomes very low, $c^*_t$ is constrained to its minimum, $\lambda M^*_t$, demonstrating the consumption drawdown constraint. In the latter part of the graph, as $V^*_t$ drops below $Z^*_t$, the system automatically injects capital, ensuring $V^*_t+A^*_t  \geq Z_t$. During this period, $c^*_t$ remains at its minimum constrained level.

We first examine the sensitivity of the optimal portfolio, the optimal consumption and the expected capital injection in Figure \ref{fig:lambda} with respect to the drawdown constraint parameter $\lambda$. Let us fix the model parameters $\rho=2,\,p=-0.1,\,\mu=0.01,\,\sigma=0.02,\,\beta=2,\,\mu_Z=\sigma_Z=0.05,\,z=10,m=20$ and plot the curves with $\lambda=0, 0.05, 0.1, 0.5, 1$, respectively. Being consistent with intuition, when the drawdown constraint parameter $\lambda$ tends to zero, the optimal portfolio, the optimal consumption and the expected capital injection converge to their counterparts of the optimal tracking problem with no consumption constraint (i.e., $\lambda =0$) in \cite{BHY23a}. More importantly, when the wealth level $x$ is sufficiently high, the optimal consumption in the case of $\lambda>0$ is in fact lower than the unconstrained one with $\lambda=0$, indicating that the drawdown constraint may suppress the consumption behavior in the large wealth regime as the aggressive consumption leads to a larger drawdown reference process. We also observe that
the optimal portfolio with $\lambda>0$ is  higher than the unconstrained case with $\lambda=0$ so that the drawdown constraint leads to a larger investment amount in the financial market to ensure the drawdown constraint to be sustainable. Figure \ref{fig:lambda} also shows that a higher capital injection is needed to support a larger consumption drawdown constraint.

\begin{figure}[h]
    \centering
\includegraphics[width=8cm]{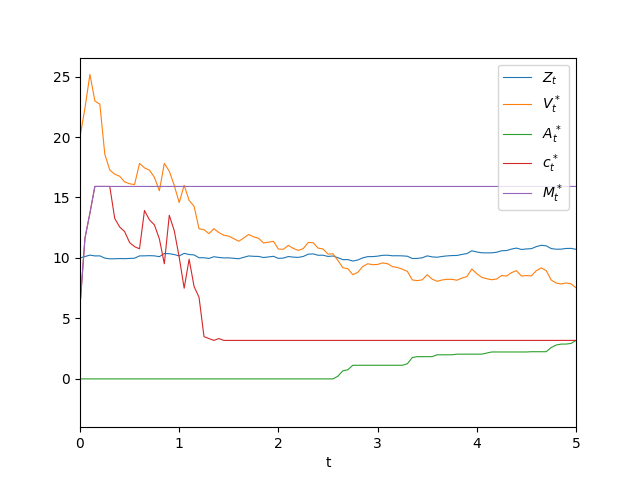}
\caption{{\small Simulation of sample path}}\label{simulation}
\end{figure}

\begin{figure}[h]
    \centering
    \begin{tabular}{c@{\extracolsep{\fill}}c@{\extracolsep{\fill}}c}
            \includegraphics[width=0.34\linewidth]{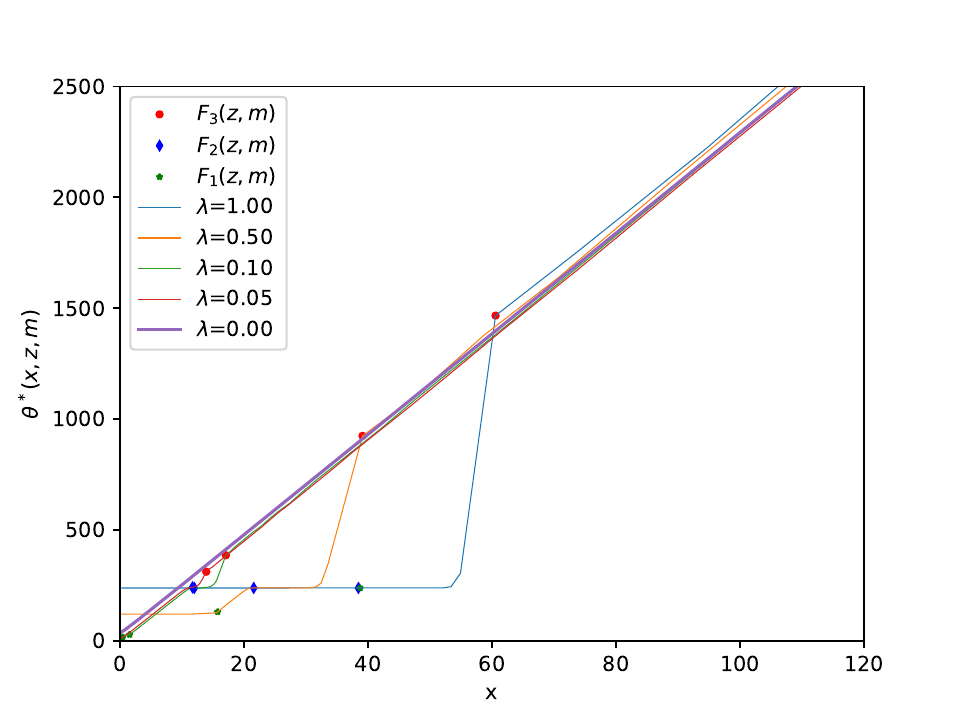} &
            \includegraphics[width=0.34\linewidth]{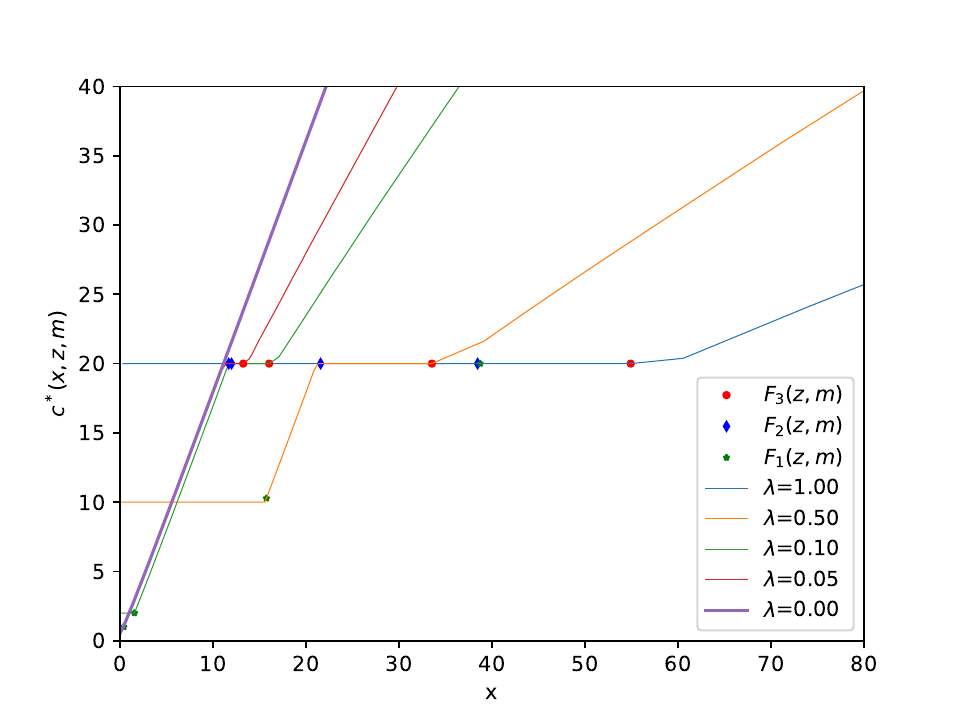}&
            \includegraphics[width=0.34\linewidth]{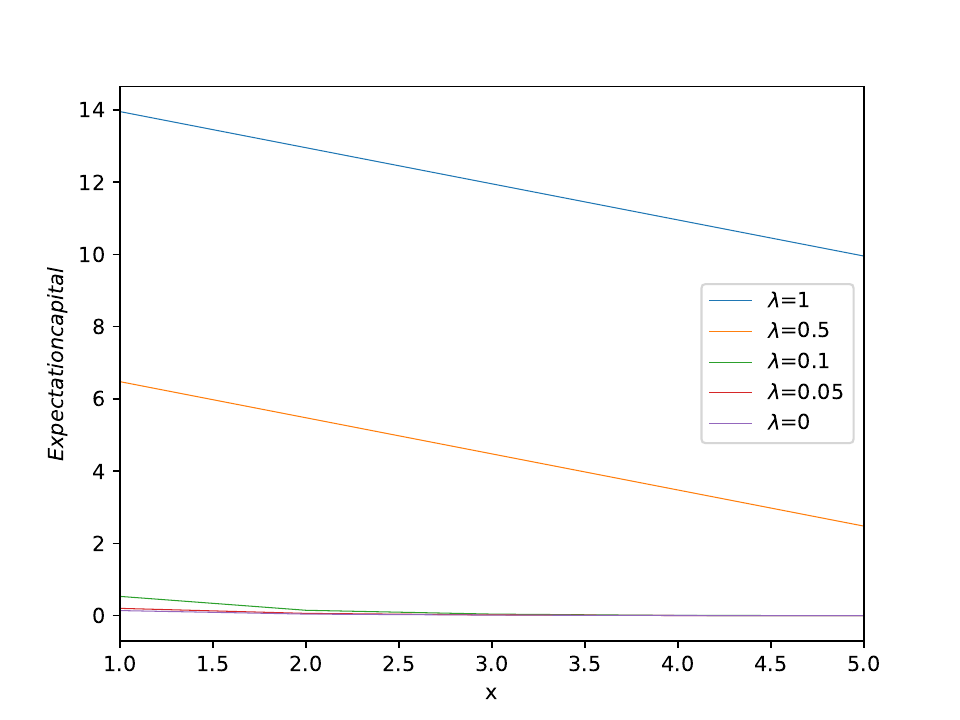}\\
            {\small(a) Optimal portfolio} & {\small(b) Optimal consumption} &{\small(c) Expected capital injection}\\
    \end{tabular}
    \caption{{\small Sensitivity results w.r.t. $\lambda$}}
    \label{fig:lambda}
 \end{figure}

We next analyze the sensitivity results with respect to the capital injection cost parameter $\beta$ in Figure \ref{fig:beta}. Let us fix model parameters $\rho=2,\,p=-0.1,\,\mu=0.01,\,\sigma=0.02,\,\lambda=0.2,\,\mu_Z=\sigma_Z=0.5,\,z=20,m=6$ and plot the optimal portfolio, the optimal consumption and the expected capital injection with varying $\beta=2, 4, 30, 40, 50$. It is not surprising to see in panel (c) of Figure \ref{fig:beta} that the larger 
capital cost parameter $\beta$, the less the capital injection. As the cost parameter $\beta$ increases, the fund manager tends to choose a smaller consumption to reduce the required capital injection. Meanwhile, the fund manager will strategically reduce the investment in the risk assets to avoid the unnecessary capital injection caused by the volatility of the controlled wealth process.

\begin{figure}[h]
    \centering
    \begin{tabular}{c@{\extracolsep{\fill}}c@{\extracolsep{\fill}}c}
            \includegraphics[width=0.34\linewidth]{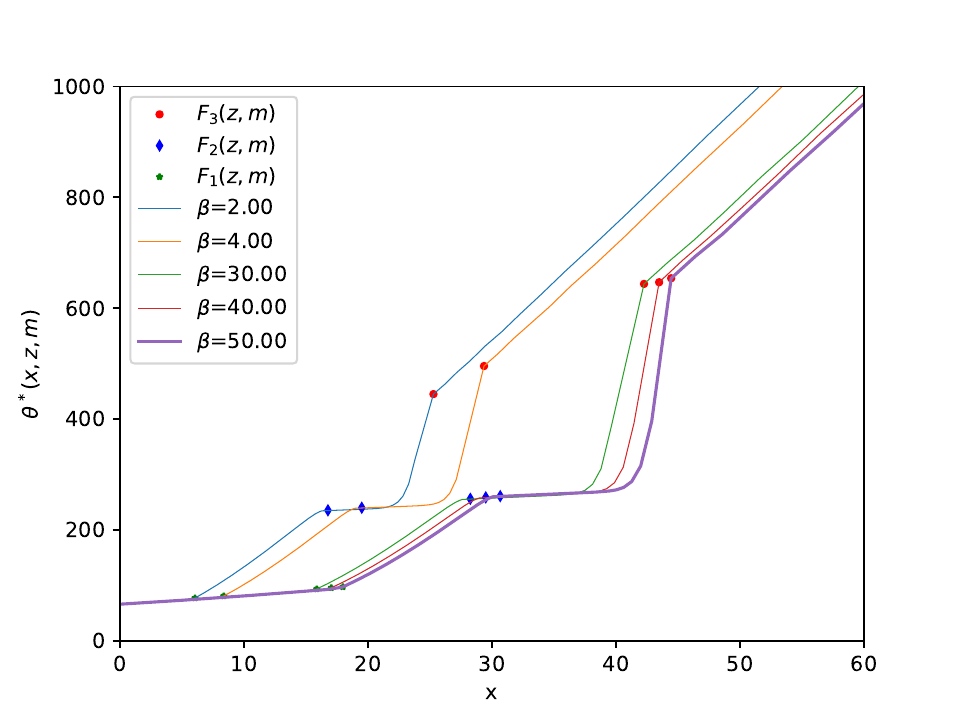} &
            \includegraphics[width=0.34\linewidth]{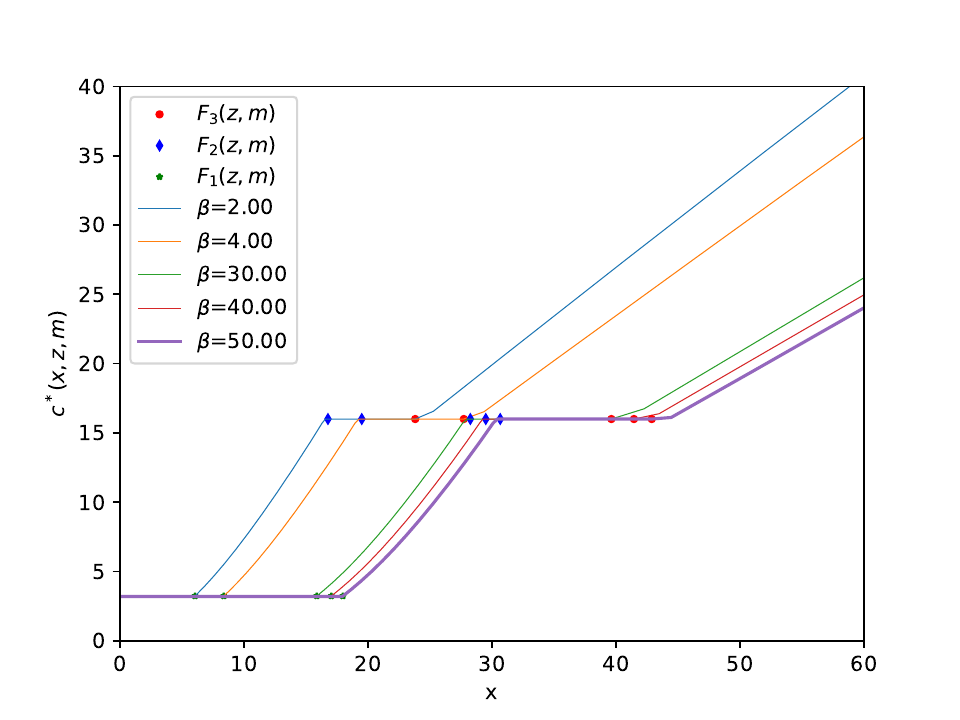}&
            \includegraphics[width=0.34\linewidth]{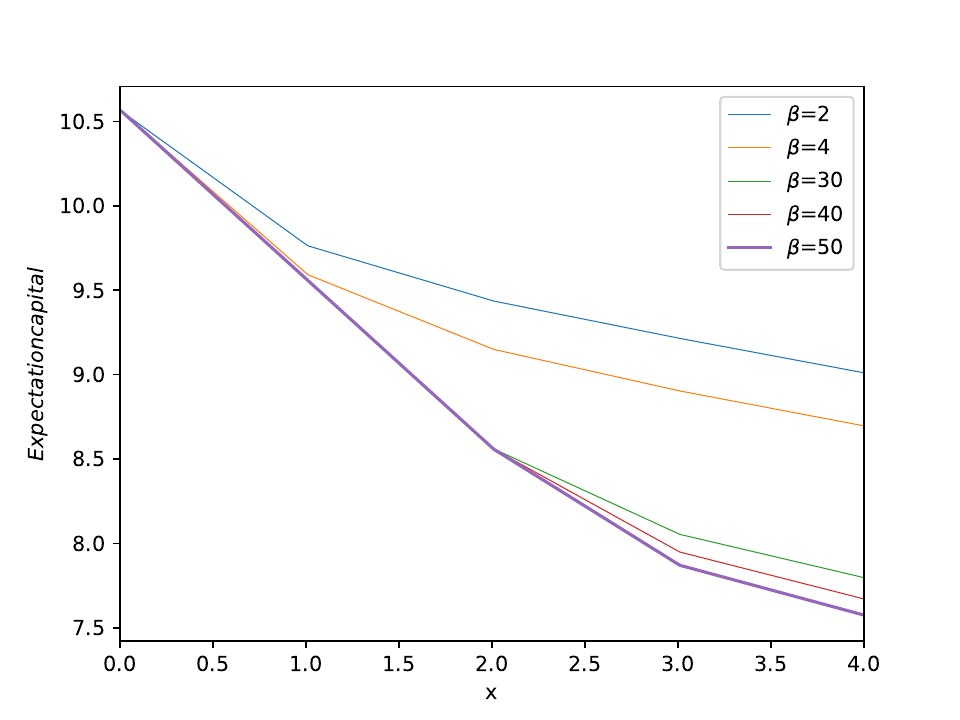} \\
            {\small(a) Optimal portfolio} & {\small(b) Optimal consumption} & {\small(c) Expected capital injection}\\
    \end{tabular}
    \caption{{\small Sensitivity results w.r.t. $\beta$.}}
    \label{fig:beta}
 \end{figure}

To understand how the market performance affects the optimal decision in our formulation, we also plot the sensitivity results w.r.t. the excessive returns $\mu=0.004, 0.008, 0.012, 0.016$ in Figure \ref{fig:mu} while fixing other model parameters $\rho=2,\,p=-0.1,\,\sigma=0.02,\,\lambda=0.2,\,\beta=2,\,\mu_Z=\sigma_Z=0.5,\,z=20,~ m=6$. From the panel (a) of Figure \ref{fig:mu}, it can be observed that the better the market performs, the more wealth the fund manager is willing to allocate into the market. It is also interesting to see from panel (b) that a higher excessive return $\mu$ results in a larger consumption plan, which is opposite of the result in the classical Merton's problem. This new phenomenon can be explained by the fact that the flexibility in capital injection may increase risk taking attitude of the agent. Particularly, when the market return is good, the necessary amount of capital injection to fulfil the benchmark tracking constraint is significantly reduced. The injected capital might be mainly used to support the more aggressive consumption behavior. Comparing with Merton's formulation under the possible bankruptcy restriction, the capital injection will incentivize the agent to spend more gains from the financial market on consumption when the market performance is good because the agent can strategically inject capital to lift up the wealth whenever it falls down a threshold.

Figure \ref{fig:theta/x} presents the sensitivity results of the optimal portfolio share $\theta^*/X$. Let us fix model parameters $\rho=2,\,p=-0.1,\,\mu=0.01,\,\sigma=0.02,\,\beta=2,\,\mu_Z=\sigma_Z=0.5,\,z=10,m=20$. As $\lambda$ approaches zero, the optimal portfolio share converges to the unconstrained case, aligning with the intuition. Notably, when $\lambda>0$ and the wealth level is relatively low, the investment share is actually higher than in the unconstrained case. This indicates that a stronger drawdown constraint may necessitate a more aggressive investment strategy to generate sufficient returns to sustain the consumption. 
At moderate wealth levels, the optimal portfolio share tends to be lower. In this ``comfortable" zone, the agent has enough wealth to meet consumption needs without frequent capital injection or severe drawdown violation. Thus, there is less incentive to take excessive risks. The strategy here focuses more on prudent growth, balancing the utility from consumption with the cost of potential capital injections. Excessive risk-taking could push the wealth into a zone where costly injections are imminent.
 When wealth is sufficiently high, the optimal portfolio share tends to stabilize and remain relatively constant. In this regime, the fund manager possesses a substantial buffer, significantly reducing the risk of breaching the benchmark tracking constraint or the consumption drawdown constraint. With ample capital, the focus shifts towards maximizing long-term growth. The fund manager can afford to take on a consistent level of market risk, as potential losses would have a comparatively minor impact on their overall financial well-being and their ability to meet future obligations. 
When $\beta$ increases, as the capital injection is more costly, the fund manager tends to avoid the possibilities of the capital injection. Finally, a larger $\mu$ generally would encourage the fund manager to allocate more wealth into the risky assets (higher $\theta^*/X$). 

\begin{figure}[h]
    \centering
    \begin{tabular}{c@{\extracolsep{\fill}}c@{\extracolsep{\fill}}c}
            \includegraphics[width=0.34\linewidth]{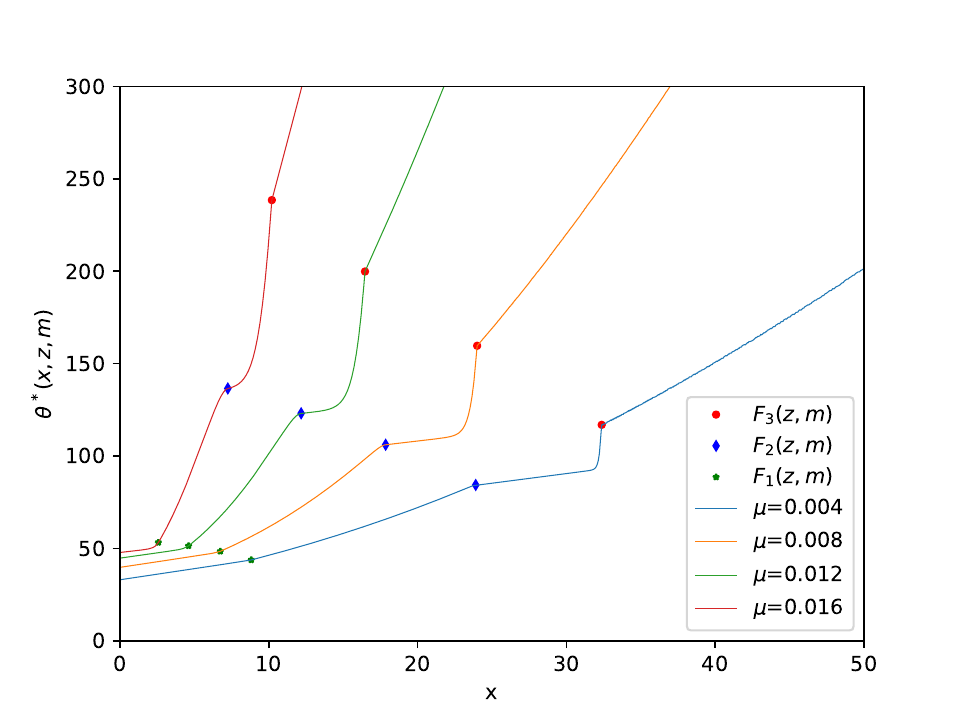} &
            \includegraphics[width=0.34\linewidth]{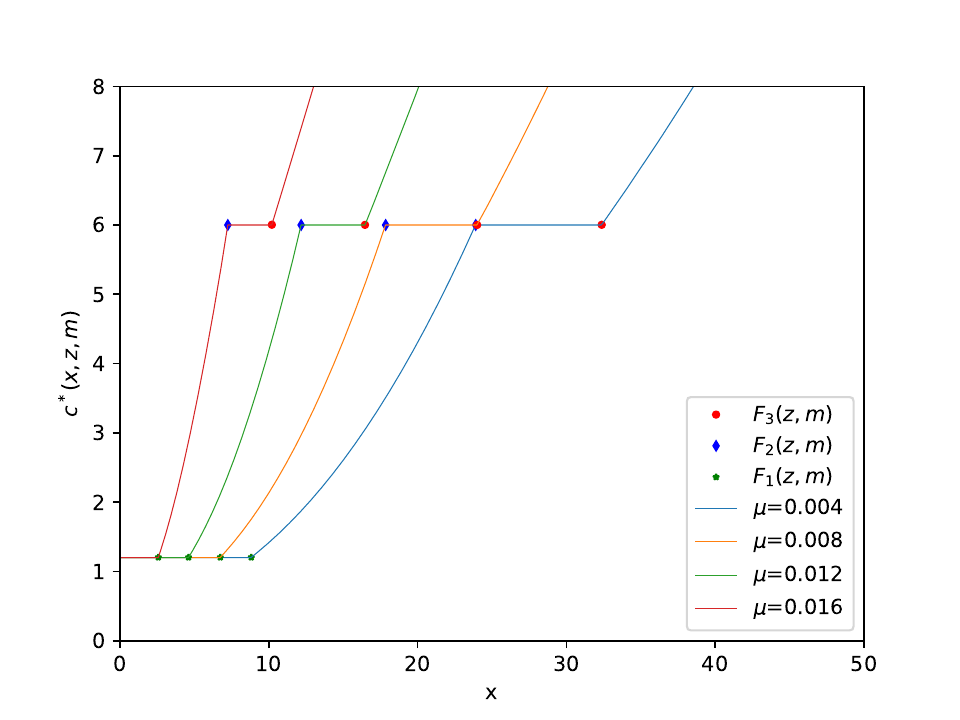} &
            \includegraphics[width=0.34\linewidth]{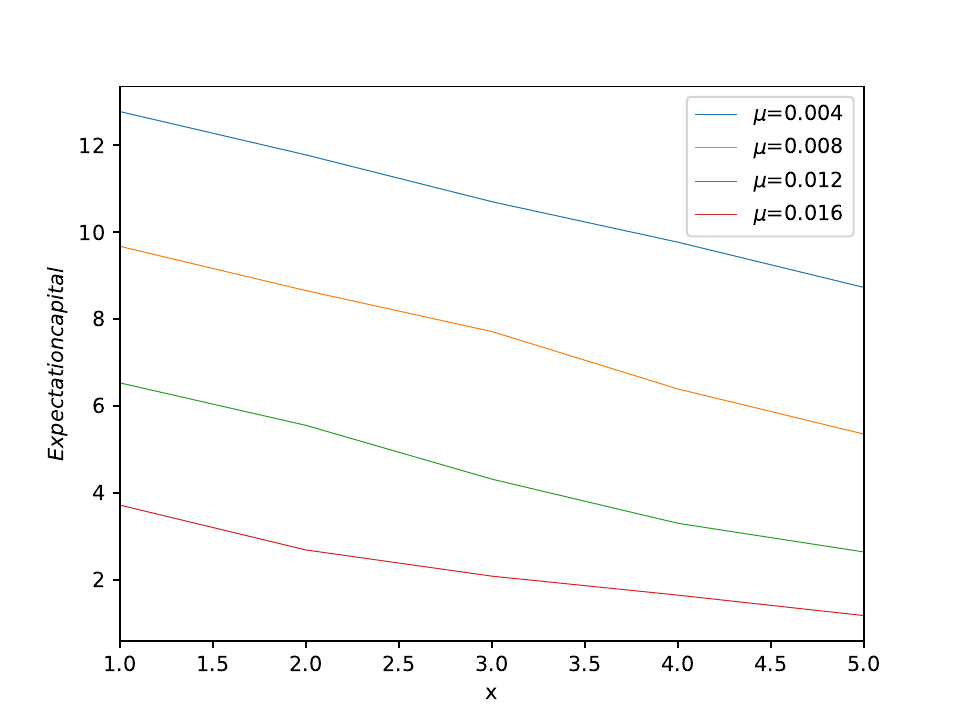}\\
            {\small (a) Optimal portfolio} & {\small (b) Optimal consumption} & {\small (c) Expected capital injection}\\
    \end{tabular}
    \caption{{\small Sensitivity results w.r.t. $\mu$.}}
    \label{fig:mu}
 \end{figure}

\begin{figure}[h]
    \centering
    \begin{tabular}{c@{\extracolsep{\fill}}c@{\extracolsep{\fill}}c}
            \includegraphics[width=0.34\linewidth]{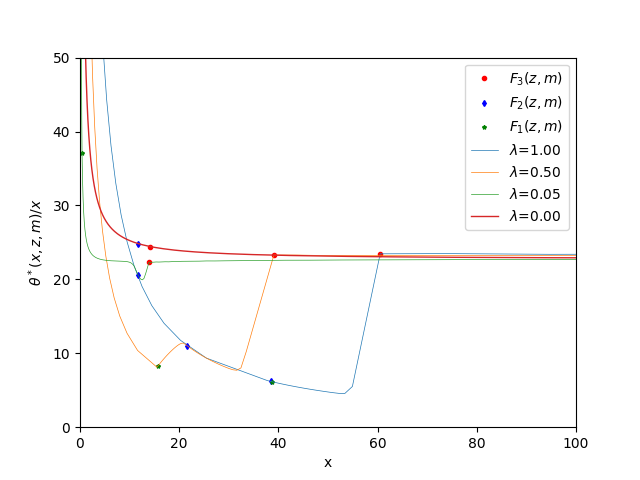} &
            \includegraphics[width=0.34\linewidth]{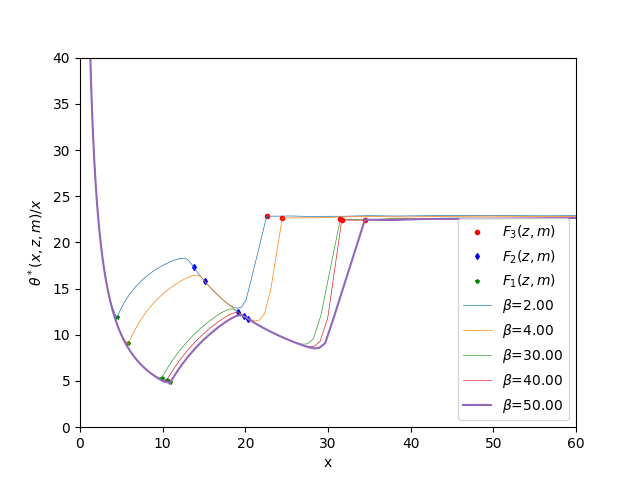}&
            \includegraphics[width=0.34\linewidth]{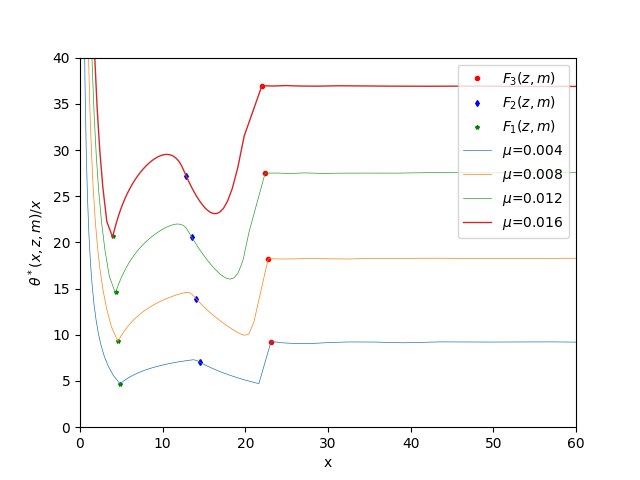} \\
            {\small(a) Sensitivity results w.r.t. $\lambda$} & {\small(b) Sensitivity results w.r.t. $\beta$} & {\small(c) Sensitivity results w.r.t.$\mu$}\\
    \end{tabular}
    \caption{{\small Sensitivity results of $\theta^*/x$.}}
    \label{fig:theta/x}
 \end{figure}

\section{Conclusions}\label{sec:conslusion}

In this paper, we study an optimal consumption problem with both benchmark tracking constraint and consumption drawdown constraint with fictitious capital injection, leading to a stochastic control problem with state-control constraints. By introducing an auxiliary reflected process, we transform it into a three-dimensional control problem with reflected dynamics. The main contributions are deriving the closed-form solution to the HJB-VI with mixed boundary conditions via domain decomposition and duality transform, and developing novel verification arguments using a constructed dual reflected process. Numerical results supplement the theory and  provide practical insights into how fund manager adjust her portfolio and consumption strategy when she concerns the benchmark process and the past maximum of consumption.  Future research could extend this framework in two directions. Firstly, it is interesting to generalize the current setting to general utility functions and benchmark processes and incomplete market models. Secondly, one may also consider singular controls in the current framework such as the proportional transaction cost for the portfolio management.

\section{Proofs}\label{sec:proofs}
 This section collects all proofs of main results in previous sections.
\begin{proof}[Proof of Proposition~\ref{prop:sol-v}]
We first prove item (i). For fixed $m\geq \frac{1}{\lambda} \beta^{\frac{1}{p-1}}$, let us define the mapping $y\mapsto F_m(y)$ that, for $y\in(0,m^{p-1}]$,
\begin{align}\label{eq:F} 
F_m(y):=\frac{\beta m^{p-1}}{(\alpha+\rho)y}+\frac{\beta}{\alpha+\rho}\ln\left(\frac{y}{\beta}\right)+\frac{\rho}{(\alpha+\rho)^2}\beta^{-\frac{\rho}{\alpha}} y^{\frac{\alpha+\rho}{\alpha}}-\frac{m^{p-1}}{\alpha+\rho} \beta^{-\frac{\rho}{\alpha}}y^{\frac{\rho}{\alpha}}.
\end{align}
Then, it holds that
\begin{align}\label{eq:F1}
F_m'(y)=\frac{\beta}{(\alpha+\rho)y^2}\left(1+\frac{\rho}{\alpha}\beta^{-\frac{\rho}{\alpha}}y^{\frac{\rho}{\alpha}}\right)(m^{p-1}-y)<0, \quad  \forall y\in(0,m^{p-1}),
\end{align}
which yields that $y\to F_m(y)$ is strictly decreasing on $(0, m^{1-p}]$. Consequently
\begin{align}
&\max_{y\in(0, m^{p-1}]}F_m(y)=\lim_{y\to 0}F_m(y)=+\infty,\quad
\min_{y\in(0, m^{p-1}]}F_m(y)=F_m( m^{p-1}).
\end{align}
Denote by $G(m)$ the term on the right side of Eq. \eqref{eq:r-1}. Then, it is sufficient to prove that 
$$G(m)\in [F_m(m^{p-1}),\infty),\quad \forall m\geq \frac{1}{\lambda} \beta^{\frac{1}{p-1}}.$$
Note that the following equivalence holds that
\begin{align}\label{eq:r1}
&F_m( m^{p-1})\leq G(m)\nonumber\\
&\quad\Longleftrightarrow\frac{\beta\lambda}{\alpha+\rho}\ln(\beta m^{1-p})+\frac{\alpha\beta^{-\frac{\rho}{\alpha}}}{(\alpha+\rho)^2}\lambda^{\frac{\alpha p-(1-p)\rho}{\alpha}}m^{-\frac{(\alpha+\rho)(1-p)}{\alpha}}+\frac{\beta(1-p)(\lambda-1)}{\alpha+\rho}+\frac{(1-p)\lambda \beta \ln(\lambda)}{\alpha+\rho}\nonumber\\
&\qquad\qquad -\frac{(1-\lambda) \beta}{\alpha+\rho}-\frac{\beta \alpha}{(\alpha+\rho)^2}+\frac{(1-p)^2\beta(\lambda-1)}{p(\alpha+\rho)}-\frac{\beta(\alpha+\rho+p\alpha)(\lambda-1)}{p(\alpha+\rho)^2}\geq 0.
\end{align}
We thus introduce the function $m\mapsto H(m)$ that
\begin{align*}
H(m):=\frac{\beta\lambda}{\alpha+\rho}\ln(\beta m^{1-p})+\frac{\alpha\beta^{-\frac{\rho}{\alpha}}}{(\alpha+\rho)^2}\lambda^{\frac{\alpha p-(1-p)\rho}{\alpha}}m^{-\frac{(\alpha+\rho)(1-p)}{\alpha}},\quad \forall m\geq\frac{1}{\lambda}\beta^{\frac{1}{p-1}}.
\end{align*}
Hence, a direct calculation yields that
\begin{align*}
H'(m)=\frac{\lambda(1-p)\beta}{m(\alpha+\rho)}\left(1-\beta^{-\frac{\rho}{\alpha}-1}(\lambda m)^{-(1-p)(\frac{\rho}{\alpha}+1)} \right)\geq 0,\quad\forall m\geq\frac{1}{\lambda}\beta^{\frac{1}{p-1}}.
\end{align*}
This implies that the mapping $m\mapsto H(m)$ is non-decreasing. As a result, we have
\begin{align}\label{eq:r2}
&\frac{\beta\lambda}{\alpha+\rho}\ln(\beta m^{1-p})+\frac{\alpha\beta^{-\frac{\rho}{\alpha}}}{(\alpha+\rho)^2}\lambda^{\frac{\alpha p-(1-p)\rho}{\alpha}}m^{-\frac{(\alpha+\rho)(1-p)}{\alpha}}+\frac{\beta(1-p)(\lambda-1)}{\alpha+\rho}+\frac{(1-p)\lambda \beta \ln(\lambda)}{\alpha+\rho}\nonumber\\
&\qquad\quad -\frac{(1-\lambda) \beta}{\alpha+\rho}-\frac{\beta \alpha}{(\alpha+\rho)^2}+\frac{(1-p)^2\beta(\lambda-1)}{p(\alpha+\rho)}-\frac{\beta(\alpha+\rho+p\alpha)(\lambda-1)}{p(\alpha+\rho)^2}\nonumber\\
&\qquad\geq \frac{(p-1)\beta\lambda \ln\lambda}{\alpha+\rho} +\frac{\lambda \beta \alpha}{(\alpha+\rho)^2}+\frac{\beta(1-p)(\lambda-1)}{\alpha+\rho}+\frac{(1-p)\lambda \beta \ln(\lambda)}{\alpha+\rho}\nonumber\\
&\qquad\quad  -\frac{(1-\lambda) \beta}{\alpha+\rho}-\frac{\beta \alpha}{(\alpha+\rho)^2}+\frac{(1-p)^2\beta(\lambda-1)}{p(\alpha+\rho)}-\frac{\beta(\alpha+\rho+p\alpha)(\lambda-1)}{p(\alpha+\rho)^2}=0.
\end{align}
It follows from \eqref{eq:r1} and \eqref{eq:r2} that the mapping $m\mapsto y^*(m)$ is well-defined.

To see that $m\to y^*(m)$ is strictly decreasing, we only discuss the case when $ m\geq  \frac{1}{\lambda}\beta^{\frac{1}{p-1}}$ as the proof of the case $\beta^{\frac{1}{p-1}}\leq m< \frac{1}{\lambda}\beta^{\frac{1}{p-1}}$ is similar. Taking the derivative w.r.t. the variable $m$ on both sides of Eq.~\eqref{eq:r-1}, we get that, for all $m>  \frac{1}{\lambda}\beta^{\frac{1}{p-1}}$,
 \begin{align}\label{eq:r-increasing-1}
&\frac{1}{y^*(m)}\frac{dy^*(m)}{dm}\left[-\frac{\beta m^{p-1}}{(\alpha+\rho)y^*(m)}+\frac{\beta}{\alpha+\rho}+\frac{\rho}{\alpha(\alpha+\rho)}\beta^{-\frac{\rho}{\alpha}} (y^*(m))^{\frac{\alpha+\rho}{\alpha}}-\frac{\rho m^{p-1}}{\alpha(\alpha+\rho)} \beta^{-\frac{\rho}{\alpha}}(y^*(m))^{\frac{\rho}{\alpha}}\right]\nonumber\\
&\quad\geq \frac{(1-p)m^{p-2}}{\alpha+\rho}\left(\beta m^{1-p}-\beta^{-\frac{\rho}{\alpha}}m^{-\frac{\rho(1-p)}{\alpha}}\right)+\frac{(\lambda-1)(1-p)\beta}{\alpha+\rho}\frac{1}{m}\\
&\qquad-\frac{1-p}{\alpha+\rho}\beta^{-\frac{\rho}{\alpha}}\left(\lambda^{\frac{\alpha p-\rho+\rho p}{\alpha}}-1\right)m^{-\frac{(\alpha+\rho)(1-p)}{\alpha}-1}=\frac{\lambda \beta(1-p)}{m(\alpha+\rho)}\left[1-\beta^{-\frac{\alpha+\rho}{\alpha}}(\lambda m)^{-\frac{(\alpha+\rho)(1-p)}{\alpha}}\right]> 0,\nonumber
\end{align}
where the first inequality follows from $y^*(m)\leq  m^{1-p}$, and the last inequality holds since $m>\beta^{\frac{1}{p-1}}/\lambda$. By using \eqref{eq:F} and \eqref{eq:F1}, one gets that, for all $m>  \frac{1}{\lambda}\beta^{\frac{1}{p-1}}$,
\begin{align}\label{eq:r-increasing-2}
-\frac{\beta m^{p-1}}{(\alpha+\rho)y^*(m)}+\frac{\beta}{\alpha+\rho}+\frac{\rho}{\alpha(\alpha+\rho)}\beta^{-\frac{\rho}{\alpha}} (y^*(m))^{\frac{\alpha+\rho}{\alpha}}-\frac{\rho m^{p-1}}{\alpha(\alpha+\rho)} \beta^{-\frac{\rho}{\alpha}}(y^*(m))^{\frac{\rho}{\alpha}}<0.
\end{align}
As a result, the estimates \eqref{eq:r-increasing-1} and \eqref{eq:r-increasing-2} yield that $\frac{dy^*(m)}{dm}<0$, and hence $m\mapsto y^*(m)$ is strictly decreasing. 

Next, we show that $\lim_{m\to\infty}y^*(m)=0$ by contradiction. Assume instead that $\lim_{m\to\infty}y^*(m)=C>0$. Sending $m\rightarrow\infty$ on both sides of Eq.~\eqref{eq:r-1}, we get that, the left hand side of Eq.~\eqref{eq:r-1} tends to 
\begin{align*}
\frac{\beta}{\alpha+\rho}\ln\left(\frac{C}{\beta}\right)+\frac{\rho}{(\alpha+\rho)^2}\beta^{-\frac{\rho}{\alpha}} C^{\frac{\alpha+\rho}{\alpha}}<+\infty.
\end{align*}
However, the right side of Eq.~\eqref{eq:r-1} goes to infinity, which yields a contradiction.

Next, we handle item (ii). Let us consider the candidate solution to Eq.~\eqref{HJB.3} satisfying the separated form that $\hat{v}(y,z,m)=l(y,m)+z \psi(y)$. We then get that the function $(y,m)\to l(y,m)$ satisfies the PDE with Neumann boundary and free boundary conditions that
\begin{align}\label{eq:l}
\begin{cases}
\displaystyle -\rho l(y,m)+\rho yl_y(y,m)+\alpha y^2l_{yy}(y,m)+\Phi(y)=0,\\[0.8em]
\displaystyle  l_y(\beta,m)=0,~ l_m(y^*(m),m)=l_{ym}(y^*(m),m)=0,
\end{cases}
\end{align}
and the function $y\mapsto\psi(y)$ solves the ODE that
\begin{align}\label{eq:h}
 (\mu_Z -\rho) \psi(y)+(\rho-\eta) y\psi_y(y)+\alpha y^2\psi_{yy}(y)-(\mu_Z-\eta)y=0
\end{align}
with the Neumann boundary condition $\psi'(\beta)=0$. 

By solving Eq. \eqref{eq:h}, we obtain $\psi(y)=y+K_1 y^{\kappa }+K_2 y^{\hat{\kappa} }$ with constants $K_1,K_2\in\R$ that will be determined later. In addition, denote $\kappa$ and $\hat{\kappa}$ as two  roots of the quadratic equation $\alpha \kappa^2+(\rho-\eta-\alpha)\kappa+\mu_Z-\rho=0$. We look for such a solution $y\mapsto\psi(y)$ with $K_2=0$ such that the Neumann boundary condition $\psi'(\beta)=0$ holds, which yields that $K_1=- \frac{\beta^{-\kappa+1}}{\kappa}$. Thus, we arrive at $\psi(y)=y-\frac{\beta^{-\kappa+1}}{\kappa}y^{\kappa}$.

Next, we solve Eq. \eqref{eq:l}. Here, we only consider the case with $ m\geq \frac{1}{\lambda}\beta^{\frac{1}{p-1}}$, as the proof of the case $\beta^{\frac{1}{p-1}}\leq m< \frac{1}{\lambda}\beta^{\frac{1}{p-1}}$ is similar. In fact, we have
\begin{align*}
l(y,m)=
\begin{cases}
 \displaystyle   \frac{1}{\beta}C_1(m)y+\beta^{\frac{\rho}{\alpha}} C_2(m)y^{-\frac{\rho}{\alpha}}+\frac{(\lambda m)^p}{p\rho}+\frac{\lambda m}{\alpha+\rho}y\ln\left(\frac{y}{\beta}\right),\qquad (\lambda m)^{p-1}<y\leq \beta,  \\[1em]
  \displaystyle    \frac{1}{\beta}C_3(m)y+\beta^{\frac{\rho}{\alpha}}C_4(m)y^{-\frac{\rho}{\alpha}}+\frac{(1-p)^3}{p(\rho(1-p)-\alpha p)}y^{\frac{p}{p-1}},\qquad  m^{p-1}<y\leq (\lambda m)^{p-1},\\[1em]
  \displaystyle    \frac{1}{\beta}C_5(m)y+\beta^{\frac{\rho}{\alpha}}C_6(m)y^{-\frac{\rho}{\alpha}}+\frac{m^p}{p\rho}+\frac{ m}{\alpha+\rho}y\ln\left(\frac{y}{\beta}\right),\qquad y^*(m)\leq  y\leq m^{p-1},
\end{cases}
\end{align*}
where the coefficient functions $m\mapsto C_i(m)$ for $i=1,...,6$ will be determined later. First of all, it follows from the {\it smooth-fit
condition} w.r.t. the variable $r$ along $y = (\lambda m)^{p-1}$ and $y= m^{1-p}$ that
\begin{align}
&\beta^{-1}(\lambda m)^{p-1} C_1(m)+\beta^{\frac{\rho}{\alpha}}(\lambda m)^{\frac{\rho(1-p)}{\alpha}}C_2(m)+\frac{(\lambda m)^p}{p\rho}-\frac{(\lambda m)^p}{\alpha+\rho}\ln(\beta(\lambda m)^{1-p})\nonumber\\
&\qquad\qquad=\beta^{-1}(\lambda m)^{p-1} C_3(m)+\beta^{\frac{\rho}{\alpha}}(\lambda m)^{\frac{\rho(1-p)}{\alpha}}C_4(m)+\frac{(1-p)^3}{p(\rho(1-p)-\alpha p)}(\lambda m)^p;\label{eq:C-1}\\
&-\beta^{-1}(\lambda m)^{p-1} C_1(m)+\frac{\rho}{\alpha}\beta^{\frac{\rho}{\alpha}}(\lambda m)^{\frac{\rho(1-p)}{\alpha}}C_2(m)+\frac{(\lambda m)^p}{\alpha+\rho}\ln(\beta(\lambda m)^{1-p})-\frac{(\lambda m)^p}{\alpha+\rho}\nonumber\\
&\qquad\qquad=-\beta^{-1}(\lambda m)^{p-1} C_3(m)+\frac{\rho}{\alpha}\beta^{\frac{\rho}{\alpha}}(\lambda m)^{\frac{\rho(1-p)}{\alpha}}C_4(m)++\frac{(1-p)^2}{\rho(1-p)-\alpha p}(\lambda m)^p;\label{eq:C-2}\\
&\beta^{-1} m^{p-1} C_3(m)+\beta^{\frac{\rho}{\alpha}}m^{\frac{\rho(1-p)}{\alpha}}C_4(m)+\frac{(1-p)^3}{p(\rho(1-p)-\alpha p)}m^p\nonumber\\
&\qquad\qquad=\beta^{-1} m^{p-1} C_5(m)+\beta^{\frac{\rho}{\alpha}}m^{\frac{\rho(1-p)}{\alpha}}C_6(m)+\frac{m^p}{p\rho}-\frac{m^p}{\alpha+\rho}\ln(\beta(\lambda m)^{1-p});\label{eq:C-3}\\
&-\beta^{-1} m^{p-1} C_3(m)+\frac{\rho}{\alpha}\beta^{\frac{\rho}{\alpha}} m^{\frac{\rho(1-p)}{\alpha}}C_4(m)+\frac{(1-p)^2}{\rho(1-p)-\alpha p}m^p\nonumber\\
&\qquad\qquad=-\beta^{-1} m^{p-1} C_5(m)+\frac{\rho}{\alpha}\beta^{\frac{\rho}{\alpha}} m^{\frac{\rho(1-p)}{\alpha}}C_6(m)+\frac{m^p}{\alpha+\rho}\ln(\beta m^{1-p})-\frac{m^p}{\alpha+\rho}.\label{eq:C-4}
\end{align}
Moreover, using the Neumann boundary condition $l_y(\beta,m)=0$ and free boundary conditions $l_m(y^*(m),m)=0,~l_{ym}(y^*(m),m)=0$, we arrive at
\begin{align}
-C_1(m)+\frac{\rho}{\alpha}C_2(m)-\frac{\lambda m\beta}{\alpha+\rho}=0,\label{eq:C-5}\\
\frac{1}{\beta}C_5'(m)y^*(m)+C_6'(m)\beta^{\frac{\rho}{\alpha}} (y^*(m))^{-\frac{\rho}{\alpha}}+\frac{m^p}{p\rho}+\frac{ m}{\alpha+\rho}y^*(m)\ln\left(\frac{y^*(m)}{\beta}\right)=0,\label{eq:C-6}\\
-\frac{1}{\beta}C_5'(m)y^*(m)+\frac{\rho}{\alpha}C_6'(m)\beta^{\frac{\rho}{\alpha}} (y^*(m))^{-\frac{\rho}{\alpha}}-\frac{ m}{\alpha+\rho}y^*(m)\ln\left(\frac{y^*(m)}{\beta}\right)-\frac{ m}{\alpha+\rho}y^*(m)=0.\label{eq:C-7}
\end{align} 
By using \eqref{eq:C-1}-\eqref{eq:C-7} and $\lim_{m\to\infty} C_6(m)=0$, we have that $C_6(m)$ can be expressed in terms of $y^*(m)$; and $C_1(m)-C_5(m)$ can be expressed in terms of $C_6(m)$ as follows:
{\small
\begin{align}
C_1(m)&=\frac{\alpha^2\beta^{-\frac{\rho}{\alpha}}}{(\alpha+\rho)^2(\rho(1-p)-\alpha p)}\left(\lambda^{\frac{\alpha p-(1-p)\rho}{\alpha}}-1\right)m^{\frac{\alpha p-(1-p)\rho}{\alpha}}-\frac{\lambda \beta m}{\alpha+\rho}+\frac{\rho}{\alpha}C_6(m),\label{eq:C1}\\
C_2(m)&=\frac{\alpha^3\beta^{-\frac{\rho}{\alpha}}}{\rho(\alpha+\rho)^2(\rho(1-p)-\alpha p)}\left(\lambda^{\frac{\alpha p-(1-p)\rho}{\alpha}}-1\right)m^{\frac{\alpha p-(1-p)\rho}{\alpha}}+C_6(m),\label{eq:C2}\\[0.6em]
C_3(m)&=
\begin{cases}
\displaystyle \frac{\alpha^2\beta^{-\frac{\rho}{\alpha}}}{(\alpha+\rho)^2(\rho(1-p)-\alpha p)}\left(\lambda^{\frac{\alpha p-(1-p)\rho}{\alpha}}-1\right)m^{\frac{\alpha p-(1-p)\rho}{\alpha}}+\frac{\rho}{\alpha}C_6(m)\\[1.2em]
\displaystyle \quad
+\frac{\beta m}{\alpha+\rho}\left[\frac{(\alpha+\rho-p\rho-(1-p)^2(\alpha+\rho))\lambda }{p(\alpha+\rho)}-{\lambda\ln\beta(\lambda m)^{1-p}}\right],\quad m\geq \frac{1}{\lambda}\beta^{\frac{1}{p-1}},\\[1.2em]
\displaystyle -\frac{\alpha^2\beta^{-\frac{\rho}{\alpha}}}{(\alpha+\rho)^2(\rho(1-p)-\alpha p)}m^{\frac{\alpha p-(1-p)\rho}{\alpha}}+\frac{(1-p)^2\beta^{\frac{p}{p-1}}}{\rho(1-p)-\alpha p}+\frac{\rho}{\alpha}C_6(m),\quad \beta^{\frac{1}{p-1}}\leq m< \frac{1}{\lambda}\beta^{\frac{1}{p-1}},
\end{cases}\label{eq:C3}\\[0.6em]
C_4(m)&= -\frac{\alpha^3\beta^{-\frac{\rho}{\alpha}}}{\rho(\alpha+\rho)^2(\rho(1-p)-\alpha p)}m^{\frac{\alpha p-(1-p)\rho}{\alpha}}+C_6(m),\label{eq:C4}\\[1.2em]
C_5(m)&=
\begin{cases}
\displaystyle
\frac{\beta m}{\alpha+\rho}\left[
-1+(1-\lambda)\left(
\frac{(1-p)^2(\alpha+\rho)-\alpha-\rho+p\rho}{p(\alpha+\rho)}
\right)-\lambda\ln(\beta(\lambda m)^{1-p})+\ln(\beta m^{1-p})\right]\\[0.8em]
\displaystyle \quad+\frac{\alpha^2\beta^{-\frac{\rho}{\alpha}}}{(\alpha+\rho)^2(\rho(1-p)-\alpha p)}\left(\lambda^{\frac{\alpha p-(1-p)\rho}{\alpha}}-1\right)m^{\frac{\alpha p-(1-p)\rho}{\alpha}}+\frac{\rho}{\alpha}C_6(m),\quad m\geq \frac{1}{\lambda}\beta^{\frac{1}{p-1}},\\[0.8em]
\displaystyle -\frac{\alpha^2\beta^{-\frac{\rho}{\alpha}}}{(\alpha+\rho)^2(\rho(1-p)-\alpha p)}m^{\frac{\alpha p-(1-p)\rho}{\alpha}}+\frac{(1-p)^2\beta^{\frac{p}{p-1}}}{\rho(1-p)-\alpha p}+\frac{\rho}{\alpha}C_6(m)\\[1.2em]
\displaystyle \quad+\frac{\beta m}{\alpha+\rho}\left[\frac{(1-p)^2}{p}-\frac{\alpha p+\alpha+\rho}{p(\alpha+\rho)}+\ln(\beta m^{1-p})\right],\quad \beta^{\frac{1}{p-1}}\leq m< \frac{1}{\lambda}\beta^{\frac{1}{p-1}},
\end{cases}
\label{eq:C5}\\[1.2em]
C_6(m)&=\int_m^{\infty}\left(\frac{\alpha}{\rho(\alpha+\rho)}\ell^{p-1}\beta^{-\frac{\rho}{\alpha}} (y^*(\ell))^{\frac{\rho}{\alpha}}-\frac{\alpha\beta}{(\alpha+\rho)^2}\beta^{\frac{\alpha+\rho}{\alpha}} (y^*(\ell))^{-\frac{\alpha+\rho}{\alpha}}\right)d\ell.\label{eq:C6}
\end{align}}


Finally, it remains to show that $m\mapsto C_6(m)$ is well-defined in the sense that $|C_6(m)|<\infty$ for any $m\geq \beta^{\frac{1}{p-1}}$. By using $y^*(m)\in(0,m^{p-1}]$ and Assumption $\rho>\rho_0$, we have
\begin{align}\label{eq:C6-1}
C_6(m)&\leq \int_m^{\infty}\frac{\alpha}{\rho(\alpha+\rho)}\ell^{p-1}\beta^{-\frac{\rho}{\alpha}} (y^*(\ell))^{\frac{\rho}{\alpha}}d\ell\nonumber\\
&\leq \int_m^{\infty}\frac{\alpha}{\rho(\alpha+\rho)}\beta^{-\frac{\rho}{\alpha}} \ell^{-\frac{(\rho+\alpha)(1-p)}{\alpha}}d\ell=\frac{\beta^{-\frac{\rho}{\alpha}}m^{-\frac{\rho(1-p)-p\alpha}{\alpha}}}{\rho(\alpha+\rho)(\rho(1-p)-p\alpha)}.
\end{align}
Using the facts $\frac{d y^*(m)}{d m}<0$ and $y^*(m)\leq m^{p-1}$, we obtain, for all $m\geq \beta^{\frac{1}{p-1}}$,
\begin{align*}
&\frac{d}{dm}\left[\frac{\alpha}{\rho(\alpha+\rho)}m^{p-1}\beta^{-\frac{\rho}{\alpha}} (y^*(m))^{\frac{\rho}{\alpha}}-\frac{\alpha\beta}{(\alpha+\rho)^2}\beta^{\frac{\alpha+\rho}{\alpha}} (y^*(m))^{-\frac{\alpha+\rho}{\alpha}}\right]\leq0,
\end{align*}
which, together with the fact that $\lim_{m\rightarrow\infty}y^*(m)=0$, implies that, for all $m\geq \beta^{\frac{1}{p-1}}$,
\begin{align*}
&\frac{\alpha}{\rho(\alpha+\rho)}m^{p-1}\beta^{-\frac{\rho}{\alpha}} (y^*(m))^{\frac{\rho}{\alpha}}-\frac{\alpha\beta}{(\alpha+\rho)^2}\beta^{\frac{\alpha+\rho}{\alpha}} (y^*(m))^{-\frac{\alpha+\rho}{\alpha}}
\nonumber\\
&\qquad\geq
\lim_{m\rightarrow\infty}\left[\frac{\alpha}{\rho(\alpha+\rho)}m^{p-1}\beta^{-\frac{\rho}{\alpha}} (y^*(m))^{\frac{\rho}{\alpha}}-\frac{\alpha\beta}{(\alpha+\rho)^2}\beta^{\frac{\alpha+\rho}{\alpha}} (y^*(m))^{-\frac{\alpha+\rho}{\alpha}}\right]=0.
\end{align*}
Thus, we have
\begin{align}\label{eq:C6-2}
C_6(m)&=\int_m^{\infty}\left(\frac{\alpha}{\rho(\alpha+\rho)}\ell^{p-1}\beta^{-\frac{\rho}{\alpha}} (y^*(\ell))^{\frac{\rho}{\alpha}}-\frac{\alpha\beta}{(\alpha+\rho)^2}\beta^{\frac{\alpha+\rho}{\alpha}} (y^*(\ell))^{-\frac{\alpha+\rho}{\alpha}}\right)d\ell>0.
\end{align}
From \eqref{eq:C6-1} and \eqref{eq:C6-2}, it follows that $m\mapsto C_6(m)$ is well-defined, which completes the proof.
\end{proof}



\begin{proof}[Proof Lemam~\ref{lem:dual-sol}]
Note that $\hat{v}_y(\beta,z,m)=0$ for all $(z,m)\in\R_+^2$, it suffices to show $\hat{v}_{yy}(y,z,m)>0$ for all $(y,z,m)\in [y^*(m),\beta]\times \R_+^2$. 

(i) The case $y^*(m)\leq  y\leq m^{p-1}$: In this case, it follows from Proposition \ref{prop:sol-v} that
\begin{align*}
\hat{v}_{yy}(y,z,m)=  \frac{\rho(\alpha+\rho)}{\alpha^2}\beta^{\frac{\rho}{\alpha}}C_6(m)y^{-\frac{\rho+2\alpha}{\alpha}}+\frac{ m}{(\alpha+\rho)y}+\beta^{-\kappa+1}(1-\kappa)zy^{\kappa-2}.
\end{align*}
As $\mu_Z\geq \eta$ and $\rho > \mu_Z$, we have $\kappa\in (0,1]$. Together with $C_6(m)>0$ by \eqref{eq:C6-2},  we have $\hat{v}_{yy}(y,z,m)>0$.

 (ii) The case $ m^{p-1}<y\leq (\lambda m)^{p-1}$:  It follows from Proposition \ref{prop:sol-v} that
\begin{align*}
\hat{v}_{yy}(y,z,m)&=  \frac{\rho(\alpha+\rho)}{\alpha^2}\beta^{\frac{\rho}{\alpha}}C_4(m)y^{-\frac{\rho+2\alpha}{\alpha}}+\frac{1-p}{\rho(1-p)-\alpha p}y^{\frac{2-p}{p-1}}+\beta^{-\kappa+1}(1-\kappa)zy^{\kappa-2}\nonumber\\
&\geq \frac{\rho(\rho+\alpha)}{\alpha^2}\beta^{\frac{\rho}{\alpha}}C_6(m)y^{-\frac{\rho+2\alpha}{\alpha}}+\frac{y^2}{\rho(1-p)-\alpha p}\left((1-p)y^{\frac{p}{p-1}}-\frac{\alpha}{\alpha+\rho}y^{\frac{\rho}{\alpha}}m^{\frac{\alpha p-(1-p)\rho}{\alpha}}\right)\nonumber\\
&=\frac{\rho(\rho+\alpha)}{\alpha^2}C_6(m)e^{\frac{\rho}{\alpha}r}+\frac{m^p y^2}{\alpha+\rho}>0.
\end{align*}

(iii) The case $(\lambda m)^{p-1}<y\leq \beta$: Note that the following estimate holds true:
 \begin{align*}
C_2(m)&=\frac{\alpha^3\beta^{-\frac{\rho}{\alpha}}}{\rho(\alpha+\rho)^2(\rho(1-p)-\alpha p)}\left(\lambda^{\frac{\alpha p-(1-p)\rho}{\alpha}}-1\right)m^{\frac{\alpha p-(1-p)\rho}{\alpha}}+C_6(m)>0.
\end{align*}
We deduce from Proposition \ref{prop:sol-v} that
\begin{align*}
   \hat{v}_{yy}(y,z,m)&=   \frac{\rho(\alpha+\rho)}{\alpha^2}\beta^{\frac{\rho}{\alpha}}C_2(m)y^{-\frac{\rho+2\alpha}{\alpha}}+\frac{\lambda m}{(\alpha+\rho)y}+\beta^{-\kappa+1}(1-\kappa)zy^{\kappa-2}>0.
\end{align*}
Putting all the pieces together, we get the desired result.
\end{proof}

\begin{proof}[Proof of Lemma \ref{lem:boundary-m}]
Note that, for all $(z,m)\in\R_+\times [\beta^{\frac{1}{p-1}},\infty)$, it holds that
\begin{align}\label{eq:F3-1}
\frac{\partial F_3(z,m)  }{\partial m}=-\hat{v}_{yy}(y^*(m),z,m)\frac{d y^*(m)}{dm}-\hat{v}_{ym}(y^*(m),z,m).
\end{align}
It follows from Proposition~\ref{prop:sol-v} and Lemma \ref{lem:dual-sol} that, for all $(z,m)\in\R_+\times [\beta^{\frac{1}{p-1}},\infty)$,
\begin{align}\label{eq:F3-2}
\hat{v}_{yy}(y^*(m),z,m)>0, \quad \frac{d y^*(m)}{dm}<0, \quad \hat{v}_{ym}(y^*(m),z,m)=0.
\end{align}
Then, we deduce from \eqref{eq:F3-1} and \eqref{eq:F3-2} that 
$\frac{\partial F_3(z,m)  }{\partial m}=-\hat{v}_{yy}(y^*(m),z,m)\frac{d y^*(m)}{dm}>0$, which yields that $m\to F_3(z,m)$ is strictly increasing. Thus, $x\mapsto m^*(x,z)$ as the inverse function of $m\mapsto F_3(z,m)$ is well-defined, and is strictly increasing in its first augment.
Note that, when $(y^*(m))^{\frac{1}{p-1}}\geq m\geq \frac{1}{\lambda}\beta^{\frac{1}{p-1}}$, we obtain 
\begin{align}\label{eq:F3-3}
x&=-\frac{1}{\beta}C_5(m)+\frac{\rho}{\alpha}\beta^{\frac{\rho}{\alpha}}C_6(m)(y^*(m))^{-\frac{\alpha+\rho}{\alpha}}-\frac{m}{\alpha+\rho}(\ln\beta^{-1}y^*(m)+1)-z(1-\beta^{1-\kappa}(y^*(m))^{\kappa-1})\nonumber\\
&\geq\frac{\rho}{\alpha}C_6(m)\beta^{\frac{\rho}{\alpha}}\left((y^*(m))^{-\frac{\alpha+\rho}{\alpha}}-\beta^{-\frac{\alpha+\rho}{\alpha}}\right)+(1-\lambda)
\frac{(1-p)\rho+(2-p)\alpha}{(\alpha+\rho)^2}m\nonumber\\
&\quad-\frac{\alpha^2\beta^{-\frac{\rho+\alpha}{\alpha}}}{(\alpha+\rho)^2(\rho(1-p)-\alpha p)}\left(\lambda^{\frac{\alpha p-(1-p)\rho}{\alpha}}-1\right)m^{\frac{\alpha p-(1-p)\rho}{\alpha}}+\frac{ m}{\alpha+\rho}\lambda\ln(\beta(\lambda m)^{1-p})\nonumber\\
&\geq \frac{ \lambda}{\alpha+\rho}m\ln(\beta m^{1-p})-\frac{\alpha^2\beta^{\frac{1}{p-1}}}{(\alpha+\rho)^2(\rho(1-p)-\alpha p)}.
\end{align}
In addition, we also have that, if $m^*(x,z)\leq \frac{1}{\lambda}\beta^{\frac{1}{p-1}}$, then
\begin{align}\label{eq:F3-4}
m^*(x,z)\ln(\beta (m^*(x,z))^{1-p})\leq -\frac{(1-p)\ln\lambda}{\lambda}\beta^{\frac{1}{p-1}}.
\end{align}
Hence, from \eqref{eq:F3-3}, \eqref{eq:F3-4} and the fact that $x\mapsto m^*(x,z)$ is strictly increasing.
\end{proof}

\begin{proof}[Proof of Lemma \ref{lem:valuefunc-v}]
It follows from Lemma \ref{lem:dual-sol} and Lemma \ref{lem:boundary-m} that $y\mapsto\hat{v}(y,z,m)$ is strictly convex and decreasing that satisfies $\hat{v}_y(\beta,z,m)=0$, $\hat{v}_y(y^*(m),z,m)=-F_3(z,m)$ and $x=F_3(z,m^*(x,z))$. Hence, the function $x\mapsto v(x,z,m)$ defined by \eqref{v.func.def1}-\eqref{v.func.def2} and $x\mapsto f(x,z,m)$ as the inverse function of $-\hat{v}_y(\cdot,z,m)$ are well-defined. Furthermore, by using Proposition \ref{prop:sol-v}, a direct calculation yields that $v(x,z,m)$ solves Eq. \eqref{eq:equation-v1} with the Neumann boundary condition on ${\cal D}$. On the other hand, for $(x,z,m)\in \R^3_+\backslash\mathcal{D}$, we have from \eqref{v.func.def2} that
\begin{align}\label{eq:v-xmz}
\begin{cases}
    v_m(x,z,m)=v_m(x,z,m^*(x,z))=0,\\
    v_x(x,z,m)=v_x(x,z,m^*(x,z))+v_{m}(x,z,m^*(x,z))m_x^*(x,z)=v_x(x,z,m^*(x,z))\\
   v_{xx}(x,z,m)=v_{xx}(x,z,m^*(x,z))+v_{xm}(x,z,m^*(x,z))m_x^*(x,z)=v_{xx}(x,z,m^*(x,z)).
   \end{cases}
\end{align}
Then, in a similar fashion, we also have $v_{z}(x,z,m)=v_z(x,z,m^*(x,z))$, $v_{zz}(x,z,m)=v_{zz}(x,z,m^*(x,z))$ and $v_{xz}(x,z,m)=v_{xz}(x,z,m^*(x,z))$. As $(x,z,m^*(x,z))\in {\cal D}$, thanks to \eqref{eq:equation-v1} and \eqref{eq:v-xmz}, we can conclude the desired result.
\end{proof}

\begin{proof}[Proof of Lemma \ref{lem:feedback-control}]
It follows from \eqref{consum.def} and Lemma \ref{lem:boundary-m} that $|c^*(x,z,m)|\leq M_c(1+x+m)$ for some positive constant $M_c$. In lieu of the duality representation, we have $x=-\hat{v}_y(f(x,z,m),z,m)$ and 
\begin{align}\label{theta.lip}
&|\theta^*(x,z,m)|\leq
|(\sigma\sigma^T)^{-1}\mu|\left|\frac{v_x(x,z,m)}{v_{xx}(x,z,m)}\right|+|\sigma_Z\gamma^T\sigma^{-1}|\left|\frac{zv_{xz}(x,z,m)}{v_{xx}(x,z,m)}\right|+|\sigma_Z\gamma^T\sigma^{-1}|z\\
&=
|(\sigma\sigma^T)^{-1}\mu|f(x,z,m)\hat{v}_{yy}(f(x,z,m),z,m)+|\sigma_Z\gamma^T\sigma^{-1}|\left|z\hat{v}_{yz}(f(x,z,m),z,m)\right|+|\sigma_Z\gamma^T\sigma^{-1}|z.\nonumber
\end{align}
For $y^*(m)\leq y\leq m^{p-1}$, $z>0$ and $m\geq \beta^{\frac{1}{p-1}}$, we have
\begin{align*}
-\frac{|z\hat{v}_{yz}(y,z,m)|}{\hat{v}_y(y,z,m)}
=
\frac{z(\beta^{1-\kappa}y^{\kappa-1}-1)}{-\frac{1}{\beta}C_5(m)+\frac{\rho}{\alpha}C_6(m)y^{-\frac{\alpha+\rho}{\alpha}}-\frac{m}{\alpha+\rho}\ln\frac{y}{\beta}-\frac{m}{\alpha+\rho}+z(\beta^{1-\kappa}y^{\kappa-1}-1)}\leq1,
\end{align*}
which results in $|z\hat{v}_{yz}(f(x,z,m),z,m)|\leq x$.
For $y^*(m)\leq y\leq m^{p-1}$ and $z>0$, we have
\begin{align}\label{yvyy.condi1}
-\frac{y\hat{v}_{yy}(y,z,m)}{\hat{v}_y(y,z,m)}
&\leq
\frac{\frac{(\alpha+\rho)\rho}{\alpha^2}\beta^{\frac{\rho}{\alpha}}C_6(m)y^{-\frac{\alpha+\rho}{\alpha}}+\frac{m}{\alpha+\rho}}{-\frac{1}{\beta}C_5(m)+\frac{\rho}{\alpha}C_6(m)y^{-\frac{\alpha+\rho}{\alpha}}-\frac{m}{\alpha+\rho}\ln\frac{y}{\beta}-\frac{m}{\alpha+\rho}}+\frac{z\beta^{1-\kappa}(1-\kappa)y^{\kappa-1}}{z(\beta^{1-\kappa}y^{\kappa-1}-1)}
\nonumber\\
&\leq
\frac{\frac{(\alpha+\rho)\rho}{\alpha^2}\beta^{\frac{\rho}{\alpha}}\frac{C_6(m)}{m}(y^*(m))^{-\frac{\alpha+\rho}{\alpha}}+\frac{1}{\alpha+\rho}}{-\frac{1}{\beta}\frac{C_5(m)}{m}+\frac{\rho}{\alpha}\frac{C_6(m)}{m}(m^{p-1})^{-\frac{\alpha+\rho}{\alpha}}-\frac{1}{\alpha+\rho}\ln\frac{m^{p-1}}{\beta}-\frac{1}{\alpha+\rho}}+(1-\kappa).
\end{align}
Note that, it holds that 
\begin{align}\label{yhatvyy.ineq}
& \frac{(\alpha+\rho)\rho}{\alpha^2}\beta^{\frac{\rho}{\alpha}}\frac{C_6(m)}{m}(y^*(m))^{-\frac{\alpha+\rho}{\alpha}}\nonumber\\
&\qquad=
\frac{1}{m}e^{\frac{\alpha+\rho}{\alpha}r^*(m)}\int_m^{\infty}\left(\frac{1}{\alpha\beta}u^{p-1}e^{-\frac{\rho}{\alpha}r^*(u)}-\frac{\rho}{\alpha(\alpha+\rho)}e^{-\frac{\alpha+\rho}{\alpha}r^*(u)}\right)du,
\end{align}
where $r^*(m)=\ln(\beta/y^*(m))$ for $m\geq \frac{1}{\lambda}\beta^{\frac{1}{p-1}}$ and $y^*(m)$ is the unique solution to Eq.~\eqref{eq:r-1}. Dividing by $r^*(m)$ and letting $m\rightarrow\infty$ on both sides of \eqref{eq:r-1}, we have
\begin{align}\label{eq.lim.r*}
&\lim_{m\rightarrow\infty}\frac{e^{r^*(m)-\ln m^{1-p}}+\ln m^{(1-p)\beta(1-\lambda)}}{r^*(m)}=
\lim_{m\rightarrow\infty}\frac{e^{\Tilde{r}(m)}+\ln m^{(1-p)\beta(1-\lambda)}}{\Tilde{r}(m)+\ln m^{1-p}}
\nonumber\\
&\qquad\qquad\qquad\qquad=\lim_{m\rightarrow\infty}\frac{1+e^{-\Tilde{r}(m)+\ln\ln m^{(1-p)\beta(1-\lambda)}}}{e^{-\Tilde{r}(m)}\Tilde{r}(m)+e^{-\Tilde{r}(m)+\ln\ln m^{1-p}}}
=\beta
\end{align}
with $\Tilde{r}(m):=r^*(m)-\ln m^{1-p}$ for $m\geq \frac{1}{\lambda}\beta^{\frac{1}{p-1}}$. It is easy to verify that $\lim_{m\rightarrow\infty}e^{-\Tilde{r}(m)+\ln\ln m^{1-p}}<\infty$ and $\lim_{m\rightarrow\infty}e^{-\Tilde{r}(m)+\ln\ln m^{(1-p)\beta(1-\lambda)}}<\infty$, which together with \eqref{eq.lim.r*} yields that 
\begin{align}\label{eq:lim-rstart}
 \lim_{m\rightarrow\infty}\left(r^*(m)-\ln m^{1-p}-\ln\ln m^{\beta(1-p)\lambda}\right)=0.   
\end{align}
Hence, one can find a constant $M_0\geq \frac{1}{\lambda}\beta^{\frac{1}{p-1}}$ such that, for $m\geq M_0$,
\begin{align*}
\ln m^{1-p}+\ln\ln m^{\beta(1-p)\lambda}-\frac{\alpha}{4\rho}\leq r^*(m)\leq \ln m^{1-p}+\ln\ln m^{\beta(1-p)\lambda}+\frac{\alpha}{4\rho}.    
\end{align*}
Then, it follows from \eqref{yhatvyy.ineq}  and $\rho>\rho_0$ that 
\begin{align}\label{yvyy.nume}
&\frac{(\alpha+\rho)\rho}{\alpha^2}\beta^{\frac{\rho}{\alpha}}\frac{C_6(m)}{m}(y^*(m))^{-\frac{\alpha+\rho}{\alpha}}+\frac{1}{\alpha+\rho}\leq
\frac{1}{\alpha+\rho}+\frac{1}{m}e^{\frac{\alpha+\rho}{\alpha}\left(\ln m^{1-p}+\ln\ln m^{\beta(1-p)\lambda}+\frac{\alpha}{4\rho}\right)}\nonumber\\
&\quad
\times\int_m^{\infty}\left(\frac{1}{\alpha\beta}u^{p-1}e^{-\frac{\rho}{\alpha}\left(\ln u^{1-p}+\ln\ln u^{\beta(1-p)\lambda}-\frac{\alpha}{4\rho}\right)}-\frac{\rho}{\alpha(\alpha+\rho)}e^{-\frac{\alpha+\rho}{\alpha}\left(\ln u^{1-p}+\ln\ln u^{\beta(1-p)\lambda}+\frac{\alpha}{4\rho}\right)}\right)du
\nonumber\\
&\leq
m^{\frac{(\alpha+\rho)(1-p)}{\alpha}-1}\int_m^{\infty}\left(\frac{(1-p)\lambda}{\alpha}e^{\frac{\alpha}{4\rho}+\frac{1}{2}}u^{-\frac{(\alpha+\rho)(1-p)}{\alpha}}\ln u-\frac{\rho }{\alpha(\alpha+\rho)}u^{-\frac{(\alpha+\rho)(1-p)}{\alpha}}\right)du+\frac{1}{\alpha+\rho}.
\end{align}
On the other hand, for $m\geq M_0$, we have 
\begin{align}\label{yvyy.deno}
&-\frac{1}{\beta}\frac{C_5(m)}{m}+\frac{\rho}{\alpha}\frac{C_6(m)}{m}(m^{p-1})^{-\frac{\alpha+\rho}{\alpha}}-\frac{1}{\alpha+\rho}\ln\frac{m^{p-1}}{\beta}-\frac{1}{\alpha+\rho}\nonumber\\
&\qquad=
(1-\lambda)\left(\frac{1-p}{\alpha+\rho}+\frac{\alpha}{(\alpha+\rho)^2}\right)+\frac{\lambda\ln\beta (\lambda m)^{1-p}}{\alpha+\rho}-\frac{\rho }{\alpha\beta }\frac{C_6(m)}{m}\\
&\quad\qquad-\frac{\alpha^2\beta^{-\frac{\alpha+\rho}{\alpha}}}{(\alpha+\rho)^2(\rho(1-p)-\alpha p)}\left(\lambda^{\frac{\alpha p-(1-p)\rho}{\alpha}}-1\right)m^{-\frac{(\alpha+\rho)(1-p)}{\alpha}}+\frac{\rho }{\alpha}C_6(m)m^{\frac{(\alpha+\rho)(1-p)}{\alpha}-1}.\nonumber
\end{align}
In view of the fact that $\lim_{m\rightarrow\infty}\lambda\ln\beta (\lambda m)^{1-p}=\infty$ and that
{\small\begin{align*}
&\lim_{m\rightarrow\infty}\Bigg[C_6(m)m^{\frac{(\alpha+\rho)(1-p)}{\alpha}-1}-\frac{\rho}{\alpha\beta}\frac{C_6(m)}{m}-\frac{\alpha^2\beta^{-\frac{\alpha+\rho}{\alpha}}}{(\alpha+\rho)^2(\rho(1-p)-\alpha p)}\left(\lambda^{\frac{\alpha p-(1-p)\rho}{\alpha}}-1\right)m^{-\frac{(\alpha+\rho)(1-p)}{\alpha}}\Bigg]=0,   
\end{align*}}there exists a constant $M_1\geq M_0\vee 1$ such that, for all $m\geq M_1$,
\begin{align*}
&
\frac{\lambda \ln\beta(\lambda m)^{1-p}}{\alpha+\rho}+C_6(m)m^{\frac{(\alpha+\rho)(1-p)}{\alpha}-1}-\frac{\rho}{\alpha\beta}\frac{C_6(m)}{m}
\nonumber\\
&\qquad
-\frac{\alpha^2\beta^{-\frac{\alpha+\rho}{\alpha}}}{(\alpha+\rho)^2(\rho(1-p)-\alpha p)}\left(\lambda^{\frac{\alpha p-(1-p)\rho}{\alpha}}-1\right)m^{-\frac{(\alpha+\rho)(1-p)}{\alpha}}\geq 
\frac{\alpha+\rho}{(\rho(1-p)-\alpha p)^2}.
\end{align*}
This, combining with \eqref{yvyy.deno}, implies that, for all $m\geq M_1$,
\begin{align}\label{yvyy.deno.2}
&-\frac{1}{\beta}\frac{C_5(m)}{m}+\frac{\rho}{\alpha}\frac{C_6(m)}{m}(m^{p-1})^{-\frac{\alpha+\rho}{\alpha}}-\frac{1}{\alpha+\rho}\ln\frac{m^{p-1}}{\beta}-\frac{1}{\alpha+\rho}\nonumber\\
&\qquad\qquad\geq
(1-\lambda)\left[\frac{1-p}{\alpha+\rho}+\frac{\alpha}{(\alpha+\rho)^2}\right]+
\frac{\alpha+\rho}{(\rho(1-p)-\alpha p)^2}.
\end{align}
Using \eqref{yvyy.condi1}, \eqref{yvyy.nume} and \eqref{yvyy.deno.2}, one can deduce that, for all $m\geq M_1$,
\begin{align}\label{theta.lip.condi1}
-\frac{y\hat{v}_{yy}(y,z,m)}{\hat{v}_y(y,z,m)}
&\leq \frac{1}{(1-\lambda)(1-p)}+\lambda(1-p)^2e^{\frac{\alpha}{4\rho}+\frac{1}{2}}.
\end{align}
On the other hand, for $y^*(m)\leq y\leq m^{p-1}$,  $\beta^{\frac{1}{p-1}}\leq m\leq M_1$ and $z>0$, we have 
\begin{align}\label{theta.lip.cond2}
y\hat{v}_{yy}(y,z,m)
&\leq
\frac{(\alpha+\rho)\rho}{\alpha^2}\beta^{\frac{\rho}{\alpha}}C_6(\beta^{\frac{1}{p-1}})(y^*(M_1))^{-\frac{\alpha+\rho}{\alpha}}+\frac{M_1}{\alpha+\rho}-\hat{v}_y(y,z,m)(1-\kappa).
\end{align}
We deduce from \eqref{theta.lip.condi1} and \eqref{theta.lip.cond2} that, for some  positive constant $M_o$ depending on $(\rho,\mu,\sigma,p,\lambda,\beta)$ and $\forall (x,z,m)\in\{(x,z,m)\in\R_+^3;~y^*(m)\leq v_x(x,z,m)<m^{p-1},~m\geq \beta^{\frac{1}{p-1}}\}$, it holds that
\begin{align}\label{ineq:theta1}
f(x,z,m)\hat{v}_{yy}(f(x,z,m),z,m)\leq M_o+\left(\frac{1}{(1-\lambda)(1-p)}+\lambda(1-p)^2e^{\frac{\alpha}{4\rho}+\frac{1}{2}}\right) x.
\end{align}
In what follows, let $M_o>0$ be a generic positive constant depending on $(\rho,\mu,\sigma,p,\lambda,\beta)$ that may differ from line to line. For $m^{p-1}\leq y\leq (\lambda m)^{p-1}$, $m\geq  M_0$ and $z>0$, we have
\begin{align}\label{yvyy.condi.2}
-\frac{y\hat{v}_{yy}(y,z,m)}{\hat{v}_y(y,z,m)}
&\leq
\frac{\frac{(\alpha+\rho)\rho}{\alpha^2}\beta^{\frac{\rho}{\alpha}}\frac{C_6(m)}{m}m^{\frac{(\alpha+\rho)(1-p)}{\alpha}}+\frac{1}{\alpha+\rho}}{-\frac{1}{\beta}\frac{C_3(m)}{m}+\frac{\rho}{\alpha}\beta^{\frac{\rho}{\alpha}}\frac{C_4(m)}{m}(\lambda m)^{\frac{(\alpha+\rho)(1-p)}{\alpha}}+\frac{(1-p)^2}{\rho(1-p)-\alpha p}\lambda }+1-\kappa.
\end{align}
Note that, for $m\geq  M_0$,
\begin{align*}
&-\frac{1}{\beta}\frac{C_3(m)}{m}+\frac{\rho}{\alpha}\beta^{\frac{\rho}{\alpha}}\frac{C_4(m)}{m}(\lambda m)^{\frac{(\alpha+\rho)(1-p)}{\alpha}}+\frac{(1-p)^2}{\rho(1-p)-\alpha p}\lambda 
\nonumber\\
&\quad=
-\frac{\alpha^2\beta^{-\frac{\alpha+\rho}{\alpha}}}{(\alpha+\rho)^2(\rho(1-p)-\alpha p)}(\lambda^{\frac{\alpha p-(1-p)\rho}{\alpha}}-1)m^{-\frac{(\alpha+\rho)(1-p)}{\alpha}}-\frac{((2-p)\alpha+\rho(1-p))\lambda}{(\alpha+\rho)^2}
\nonumber\\
&\qquad
-\frac{\rho}{\alpha\beta}\frac{C_6(m)}{m}-\frac{\alpha^2\beta^{-\frac{\alpha+\rho}{\alpha}}}{(\alpha+\rho)^2(\rho(1-p)-\alpha p)}(\lambda^{\frac{\alpha p-(1-p)\rho}{\alpha}}-1)+\frac{(1-p)^2\lambda}{\rho(1-p)-\alpha\rho}\nonumber\\
&\qquad
+\frac{\rho}{\alpha}\beta^{\frac{\rho}{\alpha}}\frac{C_6(m)}{m}(\lambda m)^{\frac{(\alpha+\rho)(1-p)}{\alpha}}+\frac{\lambda}{\alpha+\rho}\ln\beta(\lambda m)^{1-p},\nonumber
\end{align*}
and $\lim_{m\rightarrow\infty}C_6(m)m^{\frac{\rho(1-p)-\alpha p}{\alpha}}=0$, $\lim_{m\rightarrow\infty}\ln\beta(\lambda m)^{1-p}=\infty$. Then, there exists a constant $M_2\geq  \frac{1}{\lambda}\beta^{\frac{1}{p-1}}$ such that, for all $m\geq M_2$,
\begin{align}\label{yvyy.nume.1}
-\frac{1}{\beta}\frac{C_3(m)}{m}+\frac{\rho}{\alpha}\beta^{\frac{\rho}{\alpha}}\frac{C_4(m)}{m}(\lambda m)^{\frac{(\alpha+\rho)(1-p)}{\alpha}}+\frac{(1-p)^2}{\rho(1-p)-\alpha p}\lambda \geq 1,
\end{align}
and 
\begin{align}\label{yvyy.nume.2}
\frac{(\alpha+\rho)\rho}{\alpha^2}\beta^{\frac{\rho}{\alpha}}\frac{C_6(m)}{m}m^{\frac{(\alpha+\rho)(1-p)}{\alpha}}+\frac{1}{\alpha+\rho}\leq 1.
\end{align}
Therefore, thanks to \eqref{yvyy.condi.2}-\eqref{yvyy.nume.2}, we obtain that
\begin{align}\label{theta.lip.condi3}
-\frac{y\hat{v}_{yy}(y,z,m)}{\hat{v}_y(y,z,m)}\leq 2-\kappa,\quad \forall m\geq M_2, 
\end{align}
and 
\begin{align}\label{theta.lip.condi4}
y\hat{v}_{yy}(y,z,m)
&\leq\frac{\alpha^2}{(\alpha+\rho)^2(\rho(1-p)-\alpha p)}M_2+\frac{(\alpha+\rho)\rho}{\alpha^2}\beta^{\frac{\rho}{\alpha}}C_6(\beta^{\frac{1}{p-1}})M_2^{\frac{(\alpha+\rho)(1-p)}{\alpha}}
\nonumber\\
&\quad
+\frac{1-p}{(\rho(1-p)-\alpha p)M_2}-\hat{v}_{y}(y,z,m)(1-\kappa),\quad \forall\beta^{\frac{1}{p-1}}\leq m \leq M_2.
\end{align}
It follows from \eqref{theta.lip.condi3} and \eqref{theta.lip.condi4} that, for any $ (x,z,m)\in\{(x,z,m)\in\R_+^3:m^{p-1}\leq v_x(x,z,m)<(\lambda m)^{p-1},~m\geq \beta^{\frac{1}{p-1}}\}$,
\begin{align}\label{ineq:theta2}
f(x,z,m)\hat{v}_{yy}(f(x,z,m),z,m)\leq M_o+2 x.
\end{align}
Finally, for $(\lambda m)^{p-1}\leq y\leq \frac{\beta}{2}$, $m\geq \frac{1}{\lambda}\left(\frac{\beta}{2}\right)^{\frac{1}{p-1}}$ and $z>0$, it holds that
\begin{align}\label{eq:theta.lip.condi5}
y\hat{v}_{yy}(y,z,m)=\frac{(\alpha+\rho)\rho}{\alpha^2}\beta^{\frac{\rho}{\alpha}}C_2(m)y^{-\frac{\alpha+\rho}{\alpha}}+\frac{\lambda}{\alpha+\rho}+z\beta^{1-\kappa}(1-\kappa)y^{\kappa-1}.
\end{align}
Note that the following estimation holds: 
\begin{align*}
&\frac{\frac{(\alpha+\rho)\rho}{\alpha^2}\beta^{\frac{\rho}{\alpha}}C_2(m)y^{-\frac{\alpha+\rho}{\alpha}}+z\beta^{1-\kappa}(1-\kappa)y^{\kappa-1}}{-\frac{1}{\beta}C_1(m)+\frac{\rho}{\alpha}\beta^{\frac{\rho}{\alpha}}C_2(m)y^{-\frac{\alpha+\rho}{\alpha}}-\frac{\lambda m}{\alpha+\rho}\ln\frac{y}{\beta}+\frac{\lambda m}{\alpha+\rho}+z(\beta^{1-\kappa}y^{\kappa-1}-1)}
\nonumber\\
&\quad\leq
\frac{\frac{\alpha\left(\lambda-\lambda^{\frac{(\alpha+\rho)(1-p)}{\alpha}}\right)}{(\alpha+\rho)(\rho(1-p)-\alpha p)}+\frac{(\alpha+\rho)\rho}{\alpha^2}\beta^{\frac{\rho}{\alpha}}\frac{C_6(m)}{m}(\lambda m)^{\frac{(\alpha+\rho)(1-p)}{\alpha}}}{-\frac{1}{\beta}\frac{C_1(m)}{m}+\frac{\rho}{\alpha}\beta^{\frac{\rho}{\alpha}}\frac{C_2(m)}{m}\left(\frac{\beta}{2}\right)^{-\frac{\alpha+\rho}{\alpha}}+\frac{\lambda }{\alpha+\rho}\ln2+\frac{\lambda }{\alpha+\rho}}+\frac{1-\kappa}{1-(\beta/2)^{1-\kappa}}.
\end{align*}
In view of the fact $\lim_{m\rightarrow\infty}C_6(m)m^{\frac{\rho(1-p)-\alpha p}{\alpha}}=0$ and
\begin{align*}
&-\frac{1}{\beta}\frac{C_1(m)}{m}+\frac{\rho}{\alpha}\beta^{\frac{\rho}{\alpha}}\frac{C_2(m)}{m}\left(\frac{\beta}{2}\right)^{-\frac{\alpha+\rho}{\alpha}}+\frac{\lambda }{\alpha+\rho}\ln2+\frac{\lambda }{\alpha+\rho}
\nonumber\\
&\quad=
\frac{\alpha^2\beta^{-\frac{\alpha+\rho}{\alpha}}\left(\lambda^{\frac{\alpha p-\rho(1-p)}{\alpha}}-1\right)m^{-\frac{(\alpha+\rho)(1-p)}{\alpha}}}{(\alpha+\rho)^2(\rho(1-p)-\alpha p)}\left(\left(\frac{\beta}{2}\right)^{-\frac{\alpha+\rho}{\alpha}}-1\right)-\frac{\rho}{\alpha\beta}\frac{C_6(m)}{m}+\frac{\lambda(2+\ln2)}{\alpha+\rho},
\end{align*}
there exists a constant $M_3\geq \frac{1}{\lambda}\beta^{\frac{1}{p-1}}$ such that, for all $m\geq M_3$,
\begin{align*}
\frac{(\alpha+\rho)\rho}{\alpha^2}\beta^{\frac{\rho}{\alpha}}\frac{C_6(m)}{m}(\lambda m)^{\frac{(\alpha+\rho)(1-p)}{\alpha}}\leq \frac{\lambda}{\rho(1-p)-\alpha p},    
\end{align*}
and
\begin{align*}
-\frac{1}{\beta}\frac{C_1(m)}{m}+\frac{\rho}{\alpha}\beta^{\frac{\rho}{\alpha}}\frac{C_2(m)}{m}\left(\frac{\beta}{2}\right)^{-\frac{\alpha+\rho}{\alpha}}+\frac{\lambda }{\alpha+\rho}\ln2+\frac{\lambda }{\alpha+\rho}\geq \frac{1}{\rho(1-p)-\alpha p}.    
\end{align*}
As a consequence, for all $m\geq M_3$,
\begin{align}\label{theta.condi.6}
y\hat{v}_{yy}(y,z,m)\leq \frac{\lambda}{\alpha+\rho}-\left(2\lambda+\frac{1-\kappa}{1-(\beta/2)^{1-\kappa}}\right)\hat{v}_y(y,z,m).
\end{align}
For $(\lambda m)^{p-1}\leq y\leq \frac{\beta}{2}$, $\frac{1}{\lambda}\left(\frac{\beta}{2}\right)^{\frac{1}{p-1}}\leq m\leq M_3$ and $z>0$, 
we arrive at
\begin{align}\label{theta.condi.7}
y\hat{v}_{yy}(y,z,m)
&\leq
\frac{\lambda\left(\lambda-\lambda^{\frac{(\alpha+\rho)(1-p)}{\alpha}}\right)}{(\alpha+\rho)(\rho(1-p)-\alpha p)}M_3+\frac{(\alpha+\rho)\rho}{\alpha^2}\beta^{\frac{\rho}{\alpha}}C_6\left(\frac{1}{\lambda}\beta^{\frac{1}{p-1}}\right)(\lambda M_3)^{\frac{(\alpha+\rho)(1-p)}{\alpha}}\nonumber\\
&\quad
+\frac{\lambda}{\alpha+\rho}-\hat{v}_y(y,z,m)(1-\kappa).
\end{align}
For $\frac{\beta}{2}\leq y\leq \beta$, $m\geq \frac{1}{\lambda}\beta^{\frac{1}{p-1}}$ and $z>0$, one gets that
\begin{align}\label{theta.condi.last}
y\hat{v}_{yy}(y,z,m)
&\leq
\frac{(\alpha+\rho)\rho}{\alpha^2\beta}C_2(\frac{1}{\lambda}\beta^{\frac{1}{p-1}})2^{\frac{\alpha+\rho}{\alpha}}+\frac{\lambda}{\alpha+\rho}+z(1-\kappa).
\end{align}
In view of \eqref{theta.condi.6}-\eqref{theta.condi.last}, it holds that, for any $ (x,z,m)\in\{(x,z,m)\in\R_+^3;~(\lambda m)^{p-1}\leq v_x(x,z,m)\leq \beta,~m\geq \beta^{\frac{1}{p-1}}\}$,
\begin{align}\label{ineq:theta3}
f(x,z,m)\hat{v}_{yy}(f(x,z,m),z,m)\leq M_o+\left(2\lambda+\frac{1-\kappa}{1-(\beta/2)^{1-\kappa}}\right) x.
\end{align}
Then, by \eqref{ineq:theta1}, \eqref{ineq:theta2} and \eqref{ineq:theta3}, we deduce that $|\theta^*(x,z,m)|\leq M_{\theta}(1+x+z)$ for some positive constant $M_{\theta}$.
\end{proof}

Next, we provide some auxiliary results (Lemma \ref{lem:admissable}, \ref{lem:Yt} and \ref{lem:transversality}) that will be used to support the proof of Theorem \ref{thm:verification}

\begin{lemma}\label{lem:admissable}
Let $\mu_Z\geq \eta$ and $\rho>\rho_0$. Given the feedback control functions $\theta^*(x,z,m)$ and $c^*(x,z,m)$ in Theorem~\ref{thm:verification}, the system of reflected SDEs \eqref{eq:optimal-SDE} has a unique strong solution.
\end{lemma}

\begin{proof}
It follows from Lemma \ref{lem:valuefunc-v} that $\partial \theta^*/\partial h$ and  $\partial c^*/\partial h$ are both continuous function in $h\in\{x,z,m\}$. For $R>0$, define
\begin{align*}
K_R:=\max_{(x,z,m)\in[0,R]^3}\left\{\frac{\partial \theta^*}{\partial x}\vee\frac{\partial \theta^*}{\partial z}\vee\frac{\partial \theta^*}{\partial m}\vee\frac{\partial c^*}{\partial x}\vee\frac{\partial c^*}{\partial z}\vee\frac{\partial c^*}{\partial m}\right\}. 
\end{align*}
Then, for all $(x_1,z_1,m_1,x_2,z_2,m_2)\in[0,R]^6$,
\begin{align*}
&|\theta^*(x_1,z_1,m_1)-\theta^*(x_2,z_2,m_2)|+|c^*(x_1,z_1,m_1)-c^*(x_2,z_2,m_2)|\nonumber\\
&\qquad\quad \leq K_R\left(|x_1-x_2|+|z_1-z_2|+|m_1-m_2|\right).
\end{align*}
Consequently, we have from \eqref{eq:optimal-SDE} that $\theta^*(X_t^*,Z_t,M_t^*)=\theta^*\left(X_t^*,Z_t,\max\left\{m,\max_{s\in[0,t]}m^*(X_s^*,Z_s)\right\}\right)$ for $t\geq0$.
Fix $m\geq 0$ and $t>0$, for any $x,z\in C([0,t])$, let us define $$F(t,x,z):=\theta^*\left(x_t,z_t,\max\left\{m,\max_{s\in[0,t]}m^*(x_s,z_s)\right\}\right).$$

We then show $F(t,(x_s)_{s\in[0,t]},(z_s)_{s\in[0,t]})$ is locally Lipschitz-continuous. To this end, fix $t>0$ and $R_1>0$. Let $x,\hat{x},z,\hat{z}\in C([0,t])$ satisfy $\sup_{s\in[0,t]}|h_t|\leq R_1$ with $h\in\{x,\hat{x},z,\hat{z}\}$. Note that $(x,z)\mapsto m^*(x,z)$ given in Lemma~\ref{lem:boundary-m} is in $C^2$ in view of the definition of $m^*(x,z)$ and the smoothness of $(z,m)\mapsto F_3(z,m)$ . Then, we have $R_2:=\max_{(x,z)\in[0,R_1]^2}m^*(x,z)<\infty$ and $\tilde{K}_{R_1}:=\max_{(x,z)\in[0,R_1]^2}\left\{\frac{\partial m^*(x,z)}{\partial x}\vee\frac{\partial m^*(x,z)}{\partial z}\right\}<\infty$. Introduce $R=\max\{R_1,R_2\}$ and $C_R=\max\{K_{R},
\tilde{K}_{R_1}\}$. Then, it holds that 
\begin{align*}
&|F(t,x,z)-F(t,\hat{x},\hat{z})|\nonumber\\
&\quad\leq C_R\left\{|x_t-\hat{x}_t|+|z_t-\hat{z}_t|+\left|\max\left\{m,\sup_{s\in[0,t]}m^*(x_s,z_s)\right\}-\max\left\{m,\sup_{s\in[0,t]}m^*(\hat{x}_s,\hat{z}_s)\right\}\right|\right\}\nonumber\\
&\quad \leq 2C_R\left\{\sup_{s\in[0,t]}|x_s-\hat{x}_s|+\sup_{s\in[0,t]}|z_s-\hat{z}_s|\right\}.
\end{align*}
Fix $t>0$, and for any $x,z\in C([0,t])$, define $G(t,x,z):=c^*(x_t,z_t,\max\{m,\sup_{s\in[0,t]}m^*(x_s,z_s)\})$. Then, in a similar fashion, we can obtain the local Lipschitz continuity of $C([0,t])^2\ni(x,z)\mapsto G(t,x,z)$ uniformly in $t$. 
Then, we define the stopping time by $\tau_n:=\inf\{t\geq 0;~|X_t^*|\geq n~\text{or}~ |Z_t|\geq n\}$ with $n>0$.
By Theorem 7 in Section 3 of Chapter 5 in  \cite{Protter05},  the system of SDEs \eqref{eq:optimal-SDE} has a unique strong solution on $[0,\tau_n]$. Moreover, by Lemma \ref{lem:feedback-control} and estimate on moments of SDE, we have that, for any $t>0$,
\begin{align*}
&\Px(\tau_n< t)\leq \frac{1}{n}\Ex\left[\left(|X_{\tau_n\wedge t}^*|+|Z_{\tau_n\wedge t}|\right){\bf 1}_{\tau_n< t}\right]\leq \frac{1}{n}\Ex\left[|X_{\tau_n\wedge t}^*|+|Z_{\tau_n\wedge t}|\right]\leq \frac{C}{n}(1+x+z),
\end{align*}
where $C>0$ is a constant independent of $n$. This implies that $\Px(\tau_n<t)\to 0$ as $n\to\infty$. Thus, we can deduce that the system \eqref{eq:optimal-SDE} has a unique strong solution.
\end{proof}

\begin{lemma}\label{lem:Yt}
Let $\mu_Z\geq \eta$ and $\rho>\rho_0$. Consider the function  $v(x,z,m)$ defined by \eqref{v.func.def1}-\eqref{v.func.def2} for $(x,z,m)\in\R_+^3$. It then holds that $v(x,z,m)\in C^3(\R_+^3)$. Define the process $(Y_t)_{t\geq 0}$ by $Y_t=v_x(X_t^*,Z_t,M_t^*)$ for all $t\geq 0$. Then, $Y_t\in (0,\beta]$ is a reflected process that satisfies the following SDE with reflection:
\begin{align*}
dY_t= \rho Y_tdt-\mu^{\top}\sigma^{-1}Y_tdW_t-dL_t^Y,
\end{align*}
where the process $L=(L_t^Y)_{t\geq0}$ is a continuous and non-decreasing process (with $L_t^{Y}=0$)  which increases on the time set $\{t\geq 0;~Y_t=\beta\}$ only.
\end{lemma}
\begin{proof}
In view of Proposition \ref{prop:sol-v} and \eqref{v.func.def1}-\eqref{v.func.def2}, the function  $v(x,z,m)$ is $C^3$ in the interior of $\mathcal{D}$ and $\R^3_+\backslash\mathcal{D}$ and $C^2$ in $\R^3_+$. Moreover, for $(x,z,m)\in \R^3_+\backslash\mathcal{D}$, we have from \eqref{v.func.def2} and $v_m(x,z,m)=v_m(x,z,m^*(x,z))=0$ that
\begin{align*}
   v_{xxx}(x,z,m)=v_{xxx}(x,z,m^*(x,z))+v_{xxm}(x,z,m^*(x,z))m_x^*(x,z)=v_{xxx}(x,z,m^*(x,z)),
\end{align*}
which implies that $v_{xxx}$ is continuous in boundary of $\mathcal{D}$. By applying similar calculation to the other third order partial derivatives of function $v$, we know $v(x,z,m)\in C^3(\R_+^3)$.  For any $t>0$, using It\^o's rule to $Y_t=v_x(X_t^*,Z_t,M_t^*)$, we obtain
\begin{align}\label{eq:ito-Y-1}
Y_t&=Y_0+\int_0^{t} v_{xx}(X_{s}^*,Z_{s},M_s^*)(\theta_s^*(X_{s}^*,Z_{s},M_s^*))^{\top}\sigma dW_s+\int_0^{t}\sigma_Z Z_s (v_{xz}-v_{xx})(X_{s}^*,Z_{s},M_s^*)dW^{\gamma}_s\nonumber\\
&\quad+\int_0^{t}  v_{xx}(X_{s}^*,Z_{s},M_s^*)dL_s^X+\int_0^{t} {\cal L}^{\theta_s^*,c_s^*} v_x(X_{s}^*,Z_{s},M_s^*)ds+\int_0^{t}   v_{xm}(X_{s}^*,Z_{s},M_s^*)dM_s^{*,c}\nonumber\\
&\quad+\sum_{0<s\leq t} e^{-\rho s}(v_x(X_{s}^*,Z_{s},M^*_{s})-v_x(X_{s}^*,Z_{s},M_{s-}^*)),
\end{align}
where $M^{*,c}$ is the continuous part of $M^*$. Note that, the process $M^*$ can only jump at time $t=0$ if $m<m^*(x,z)$, then $(X^*_s,Z_s,M^*_s)$ stays in the domain ${\cal D}$ for all $s>0$. In view that $M_s^{*,c}$ increases if and only if $M_s^{*,c}=m^*(X_s^*,Z_s)$, $v_m(x,z,m^*(x,z))=0$ and \eqref{v.func.def2} holds, we deduce
\begin{align}\label{eq:ito-Y-2}
&\int_0^{t} e^{-\rho s}  v_{xm}(X_{s}^*,Z_{s},M_s^*)dM_s^{*,c}+\sum_{0<s\leq t} e^{-\rho s} (v_x(X_{s}^*,Z_{s},M_{s}^*)-v_x(X_{s}^*,Z_{s},M_{s-}^*))=0.
\end{align}
By Lemma \ref{lem:valuefunc-v}, Lemma \ref{lem:feedback-control} and $W_t^{\gamma}=\gamma^{\top}W_t$, we can obtain that
\begin{align}
&\int_0^{t} \left(v_{xx}(X_{s}^*,Z_{s},M_s^*)(\theta_s^*(X_{s}^*,Z_{s},M_s^*))^{\top}\sigma dW_s+\sigma_Z Z_s (v_{xz}-v_{xx})(X_{s}^*,Z_{s},M_s^*)dW^{\gamma}_s\right)
\nonumber\\
&\quad=-\int_0^t \mu^{\top}\sigma^{-1}v_x(X_{s}^*,Z_{s},M_s^*)dW_s=-\int_0^t \mu^{\top}\sigma^{-1}Y_sdW_s,\label{eq:ito-Y-3}\\
&\int_0^{t} {\cal L}^{\theta_s^*,c_s^*} v_x(X_{s}^*,Z_{s},M_s^*)ds=\int_0^t \rho v_x(X_{s}^*,Z_{s},M_s^*)ds=\int_0^t \rho Y_sds.\label{eq:ito-Y-4}
\end{align}
Denote $L_t^Y:=-\int_0^{t}  v_{xx}(X_{s}^*,Z_{s},M_s^*)dL_s^X$ for $t\geq 0$. Consequently, it follows from \eqref{eq:ito-Y-1}-\eqref{eq:ito-Y-4} that
\begin{align*}
dY_t= \rho Y_tdt-\mu^{\top}\sigma^{-1}Y_tdW_t-dL_t^Y,\quad t> 0.
\end{align*}
Noting that $v_x(x,z,m)\leq \beta$, $v_x(0,z,m)= \beta$ and $v_{xx}(x,z,m)\leq 0$ for all $(x,z,m)\in\R^3_+$, we have that the process $L=(L_t^Y)_{t\geq0}$ is a continuous and non-decreasing process (with $L_t^{Y}=0$)  which increases on the time set $\{t\geq 0;~Y_t=\beta\}$ only. This implies that  $Y$ taking values on $(0,\beta]$ is a reflected process and $L^Y$ is the local time process of $Y$. 
\end{proof}

\begin{lemma}\label{lem:transversality}
Let $\mu_Z\geq \eta$ and $\rho>\rho_0$. Consider the reflected process $Y=(Y_t)_{t\geq 0}$ defined in Lemma \ref{lem:Yt}. Then, we have 
\begin{align*}
\limsup_{T\to \infty}e^{-\rho T}\Ex\left[Y_T^\frac{p}{p-1}\right]=0.
\end{align*}
\end{lemma}
\begin{proof}
For the case $p<0$, the result obviously holds as $Y_T^{\frac{p}{p-1}}\leq \beta^{\frac{p}{p-1}}$ a.s., for all $T\geq 0$. In what follows, we only consider the case $p\in(0,1)$. 

For $t\geq 0$, let $H_t=\beta Y_t^{-1}$, then by using It\^o's rule to $H_t$, we can deduce that $H_t$ taking values on $[1,\infty)$ is a reflected process satisfies the following SDE with reflection:
\begin{align*}
dH_t= (\alpha-\rho) H_tdt+\mu^{\top}\sigma^{-1}H_tdW_t+dL_t^H,
\end{align*}
where the process $L=(L_t^H)_{t\geq0}$ is the local time process which is continuous and non-decreasing  with $L_t^{H}=0$  and increases on the time set $\{t\geq 0;~H_t=1\}$ only. Then we can see
\begin{align*}
0\leq \limsup_{T\to \infty}e^{-\rho T}\Ex\left[Y_T^{\frac{p}{p-1}}\right]&=\limsup_{T\to \infty}e^{-\rho T}\beta^{\frac{p}{p-1}} \Ex\left[H_T^{\frac{p}{1-p}}\right]\leq\limsup_{T\to \infty}e^{-\rho T}\beta^{\frac{p}{p-1}} \Ex\left[H_T^{\frac{p'}{1-p'}}\right],
\end{align*}
where $p':=\max\{2/3,p\}$.  If fact, $\rho\geq \max\{2\alpha,p'\alpha/(1-p')\}$ is equivalent to $\rho\geq \max\{2\alpha,p\alpha/(1-p)\}$, thus it is sufficient to deal with the case where $p\in[2/3,1)$. Noting that $h\leq 2h-2$ when $h\geq 2$, then it holds that
\begin{align*}
\limsup_{T\to \infty}e^{-\rho T}\Ex\left[Y_T^{\frac{p}{p-1}}\right]&\leq \limsup_{T\to \infty}e^{-\rho T}\beta^{\frac{p}{p-1}} \Ex\left[2^{\frac{p}{1-p}}+(2H_T-2)^{\frac{p}{1-p}}\right]\nonumber\\
&=\limsup_{T\to \infty}e^{-\rho T}\left(\frac{\beta}{2}\right)^{\frac{p}{p-1}} \Ex\left[(H_T-1)^{\frac{p}{1-p}}\right].
\end{align*}
Using It\^o's rule to $(H_T-1)^{\frac{p}{1-p}}$ and taking expectation, we obtain from $\rho\geq p\alpha/(1-p)$ that
\begin{align}\label{eq:Ito-H-1}
&\Ex\left[(H_T-1)^{\frac{p}{1-p}}\right]
\leq(H_0-1)^{\frac{p}{1-p}}+\frac{p(2p-1)\alpha}{(1-p)^2}\int_0^T \Ex\left[H_t(H_t-1)^{\frac{3p-2}{1-p}}\right]dt.
\end{align}
If $p=2/3$ (i.e., $3p-2=0$ and $p/(1-p)=2$), then the estimate \eqref{eq:Ito-H-1} becomes
\begin{align}\label{eq:Ito-H-2}
\Ex\left[(H_T-1)^{2}\right]&
\leq(H_0-1)^{2}+3\alpha T+\alpha\int_0^T \Ex\left[(H_t-1)^{2}\right]dt.
\end{align}
It follows from \eqref{eq:Ito-H-2} and the Gronwall’s lemma that, for all $T\geq 0$,
\begin{align*}
\Ex\left[(H_T-1)^{2}\right]
&\leq (H_0-1)^{2}+3\alpha T+\frac{1}{\alpha}\left( (H_0-1)^{\frac{p}{1-p}}-3\right)\left(e^{\alpha T}-1\right)+3Te^{\alpha T}.
\end{align*}
This yields that, for $\rho>\alpha p/(1-p)>\alpha$,
\begin{align*}
\limsup_{T\to \infty}e^{-\rho T}\Ex\left[Y_T^{\frac{p}{p-1}}\right]&\leq\limsup_{T\to \infty}e^{-\rho T}\left(\frac{\beta}{2}\right)^{\frac{p}{p-1}} \Ex\left[(H_T-1)^{\frac{p}{1-p}}\right]=0.
\end{align*}
On the other hand, if $p>2/3$ (i.e., $3p-2>0$), we have that, for $h\geq 1$, 
\begin{align}\label{eq:hh}
h(h-1)^{\frac{3p-2}{1-p}}&=(h-1)^{\frac{2p-1}{1-p}}+(h-1)^{\frac{3p-2}{1-p}}
\leq \max\left\{\frac{p}{1-p},h-1\right\}^{\frac{2p-1}{1-p}}+\max\left\{\frac{p}{1-p},h-1\right\}^{\frac{3p-2}{1-p}}\nonumber\\
&\leq \left(\frac{p}{1-p}\right)^{\frac{2p-1}{1-p}}+\frac{(h-1)^{\frac{p}{1-p}}}{\frac{p}{1-p}}+ \left(\frac{p}{1-p}\right)^{\frac{3p-2}{1-p}}+\frac{(h-1)^{\frac{p}{1-p}}}{\left(\frac{p}{1-p}\right)^2}\nonumber\\
&=\left(\frac{p}{1-p}\right)^{\frac{2p-1}{1-p}}+\left(\frac{p}{1-p}\right)^{\frac{3p-2}{1-p}}+\frac{1-p}{p^2}(h-1)^{\frac{p}{1-p}}.
\end{align}
It follows from \eqref{eq:Ito-H-1} and \eqref{eq:hh} that
\begin{align}\label{eq:Ito-H-3}
\Ex\left[(H_T-1)^{\frac{p}{1-p}}\right]&\leq (H_0-1)^{\frac{p}{1-p}}
+CT+\frac{(2p-1)\alpha}{1-p}\int_0^T \Ex\left[(H_t-1)^{\frac{p}{1-p}}\right]dt,
\end{align}
where $C>0$ is a constant independent of $\rho$. In a similar fashion, by using Gronwall’s lemma to \eqref{eq:Ito-H-3} and noting $\rho>p\alpha/(1-p)\geq (2p-1)\alpha/(p-1)$, we deduce that
\begin{align*}
\limsup_{T\to \infty}e^{-\rho T}\Ex\left[Y_T^{\frac{p}{p-1}}\right]&\leq\limsup_{T\to \infty}e^{-\rho T}\left(\frac{\beta}{2}\right)^{\frac{p}{p-1}} \Ex\left[(H_T-1)^{\frac{p}{1-p}}\right]=0.
\end{align*}
Thus, we complete the proof of the lemma.
\end{proof}

\begin{proof}[Proof of Theorem \ref{thm:verification}]
We first show the validity of the inequality \eqref{eq:value-func}. For any $(\theta,c)\in \mathbb{U}^r$, let $(X_t,Z_t,M_t)_{t\geq 0}$ be the corresponding state process with initial data $(x,z,m)\in \R_+^3$. Fix $T>0$. It follows from It\^o's formula that
\begin{align}\label{eq:itoveri}
&e^{-\rho T}v(X_{T},Z_{T},m)+\int_0^{T} e^{-\rho s} U(c_s)ds\nonumber\\
&\quad=v(x,z,m)+\int_0^{T}e^{-\rho s} v_{x}(X_{s},Z_{s},m)\theta_s^{\top}\sigma dW_s+\int_0^{T}e^{-\rho s}\sigma_Z Z_s (v_{z}-v_x)(X_{s},Z_{s},m)dW^{\gamma}_s\nonumber\\
&\qquad+\int_0^{T} e^{-\rho s} v_x(X_{s},Z_{s},m)dL_s^X+\int_0^{T} e^{-\rho s}({\cal L}^{\theta_s,c_s} v-\rho v)(X_{s},Z_{s},m)ds,
\end{align}
where, for $(\theta,c)\in \R^d\times \R_+$, the operator ${\cal L}^{\theta,c}$ acting on $C^2(\R_+^2)$ is defined by
\begin{align*}
{\cal L}^{\theta,c}g&:=
\theta^{\top}\mu g_x+\frac{1}{2}\theta^{\top}\sigma\sigma^{\top}\theta g_{xx}+ \theta^{\top} \sigma \gamma\sigma_Z z (g_{x z}-g_{xx}) +U(c)-cv_g\\
&\qquad-\sigma_Z^2 z^2 g_{xz}+\frac{1}{2}\sigma_Z^2 z^2 (g_{xx}+g_{zz})+\mu_Z z (g_z-g_x),\quad \forall g\in C^{2}(\R_+^2).
\end{align*}
Then, for all $(x,z,m)\in\R_+^3$ and $(\theta,c)\in\R^d\times[\lambda m,m]$,  Lemma \ref{lem:valuefunc-v} implies that $({\cal L}^{\theta,c} v-\rho v)(x,z,m)\leq 0$ and $v_x(0,z,m)=\beta$. Consequently, taking the expectation on both sides of \eqref{eq:itoveri}, we deduce
\begin{align}\label{eq:value-ineq}
\Ex\left[e^{-\rho T}v(X_{T},Z_{T},m)\right]+\mathbb{E}\left[\int_0^{T}e^{-\rho t}U(c_t)dt-\beta\int_0^{T}e^{-\rho t}dL_t^X\right]\leq v(x,z,m).
\end{align}
Using Lemma \ref{lem:valuefunc-v} again, we arrive at $v_x(x,z,m)\geq 0$ and $|v_z(x,z,m)|\leq \beta/\kappa$ for all $(x,z,m)\in \R_+^3$. Thus, one gets, for all $(x,z,m) \in \R_+^3$,
\begin{align}\label{eq:dominate-v}
v(x,z,m)\geq v(0,z,m)\geq v(0,0,m)-\frac{\beta}{\kappa}z.
\end{align}
By letting $T\to\infty$ in \eqref{eq:value-ineq}, we obtain from \eqref{eq:dominate-v}, Dominated Convergence Theorem (DCT), the Monotone
Convergence Theorem (MCT) and $\rho>\mu_Z$ that
\begin{align}\label{eq:optimal-v-1}
v(x,z,m)&\geq \mathbb{E}\left[\int_0^{\infty}e^{-\rho t}U(c_t)dt-\beta\int_0^{\infty}e^{-\rho t}dL_t^X\right]+\liminf_{T\to \infty}\Ex\left[e^{-\rho T}v(X_{T},Z_{T},m)\right]\nonumber\\
&\geq \mathbb{E}\left[\int_0^{\infty}e^{-\rho t}U(c_t)dt-\beta\int_0^{\infty}e^{-\rho t}dL_t^X\right]+\liminf_{T\to \infty}e^{-\rho T}\left\{v(0,0,m)-\frac{\beta}{\kappa}\Ex[Z_T]\right\}\nonumber\\
&=\mathbb{E}\left[\int_0^{\infty}e^{-\rho t}U(c_t)dt-\beta\int_0^{\infty}e^{-\rho t}dL_t^X\right].
\end{align}

Next, we prove that the equality in \eqref{eq:value-func} holds true when $(\theta,c)=(\theta^*,c^*)$. It follows from Lemma~\ref{lem:admissable} that $(\theta^*,c^*)\in\mathbb{U}^r$. We next show that the following transversality condition holds:
\begin{align}\label{eq:transversality}
\limsup_{T\to \infty}\Ex\left[e^{-\rho T}v(X_T^*,Z_T,M_T^*)\right]\leq 0.
\end{align}

To this end, we introduce an auxiliary dual process $(Y_t)_{t\geq 0}$ with $Y_t=v_x(X_t^*,Z_t,M_t^*)$ for all $t\geq 0$ to facilitate the proof of the above convergence. By Lemma \ref{lem:Yt}, we know that $Y_t$ taking values on $(0,\beta]$ satisfies the SDE with reflection that
\begin{align}\label{eq:Yt}
dY_t= \rho Y_tdt-\mu^{\top}\sigma^{-1}Y_tdW_t-dL_t^Y.
\end{align}
Here, the process $L=(L_t^Y)_{t\geq0}$ is a continuous and non-decreasing process (with $L_t^{Y}=0$) that increases on the time set $\{t\geq 0;~Y_t=\beta\}$ only. Moreover, it follows from the dual relationship  that
\begin{align}\label{eq:v-Yt}
v(X_t^*,Z_t,M_t^*)=\hat{v}(Y_t,Z_z,M_t^*)+X_t^*Y_t=\hat{v}(Y_t,Z_z,M_t^*)-Y_t\hat{v}_y(Y_t,Z_z,M_t^*),\quad \forall t\geq 0.
\end{align}
Note that $\frac{\partial (\hat{v}-y\hat{v}_y)}{\partial y}(y,z,m)=-\hat{v}_{yy}(y,z,m)<0$ for all $(y,z,m)\in [y^*(m),\beta]\times\R_+\times [\beta^{\frac{1}{p-1}},\infty)$. It follows that the function $\hat{v}-y\hat{v}_y$ is strictly decreasing with respect to $y$. Thus, it suffices to consider the case $y=y^*(m)\leq m^{p-1}$. Then, by Proposition \ref{prop:sol-v} and \eqref{yvyy.nume}, we have that
\begin{align}\label{eq:hatv-y}
\hat{v}(y,z,m)-y\hat{v}_y(y,z,m)
&\leq K C_6(m)y^{-\frac{\rho}{\alpha}}\leq K my\leq y^{\frac{1}{p-1}}y=K y^{\frac{p}{p-1}},
\end{align}
where $K$ is a positive constant that might be different from line to line. For the discount rate $\rho>\rho_0$, we have from \eqref{eq:v-Yt}, \eqref{eq:hatv-y} and Lemma~\ref{lem:transversality}, that
\begin{align}
\limsup_{T\to \infty}\Ex\left[e^{-\rho T}v(X_T^*,Z_T,M_T^*)\right]\leq  K \limsup_{T\to \infty}e^{-\rho T}\Ex\left[Y_T^{\frac{p}{p-1}}\right]=0,
\end{align}
which gives the desired transversality condition \eqref{eq:transversality}.  

Now, for any $T>0$, using It\^o's rule, we obtain
\begin{align}\label{eq:itoveri-2}
&e^{-\rho T}v(X_{T}^*,Z_{T},M_T^*)+\int_0^{T} e^{-\rho s} U(c_s^*)ds\nonumber\\
&~=v(x,z,m)+\int_0^{T}e^{-\rho s} v_{x}(X_{s}^*,Z_{s},M_s^*)(\theta_s^*)^{\top}\sigma dW_s+\int_0^{T}e^{-\rho s}\sigma_Z Z_s (v_{z}-v_x)(X_{s}^*,Z_{s},M_s^*)dW^{\gamma}_s\nonumber\\
&\quad+\int_0^{T} e^{-\rho s} v_x(X_{s}^*,Z_{s},M_s^*)dL_s^X+\int_0^{T} e^{-\rho s}({\cal L}^{\theta_s^*,c_s^*} v-\rho v)(X_{s}^*,Z_{s},M_s^*)ds\nonumber\\
&\quad+\int_0^{T} e^{-\rho s}  v_m(X_{s}^*,Z_{s},M_s^*)dM_s^{*,c}+\sum_{0<s\leq T} e^{-\rho s}(v(X_{s}^*,Z_{s},M^*_{s})-v(X_{s}^*,Z_{s},M_{s-}^*)),
\end{align}
where $M^{*,c}$ is the continuous part of $M^*$. Note that, the process $M^*$ can only jump at time $t=0$ if $m<m^*(x,z)$, then $(X^*_s,Z_s,M^*_s)$ stays in the domain ${\cal D}=\{(x,z,m)\in\R^3_+;x\leq F_3(z,m)\}$ for all $s>0$. Then, it follows from $M_s^{*,c}$ increases if and only if $M_s^{*,c}=m^*(X_s^*,Z_s)$, $v_m(x,z,m^*(x,z))=0$ and \eqref{v.func.def2} holds, we then deduce
\begin{align*}
&\int_0^{T} e^{-\rho s}  v_m(X_{s}^*,Z_{s},M_s^*)dM_s^{*,c}+\sum_{0<s\leq T} e^{-\rho s} (v(X_{s}^*,Z_{s},M_{s}^*)-v(X_{s}^*,Z_{s},M_{s-}^*))\\
&\qquad =v(X_{0}^*,Z_{0},M_{0}^*)-v(X_{0}^*,Z_{0},M_0^*)=0.
\end{align*}
Thus, by taking the expectation on both sides of Eq.~\eqref{eq:itoveri-2}, it follows from  $({\cal L}^{\theta^*,c^*} v-\rho v)(x,z,m)= 0$ and $v_x(0,z,m)=\beta$ for all $(x,z,m)\in{\cal D}$ that
\begin{align}\label{eq:value-eq}
 v(x,z,m)=\Ex\left[e^{-\rho T}v(X_{T}^*,Z_{T}^*,M_T^*)\right]+\Ex\left[\int_0^{T}e^{-\rho t}U(c_t^*)dt-\beta\int_0^{T}e^{-\rho t}dL_t^{X^*}\right].
\end{align}
Letting $T\to\infty$ in \eqref{eq:value-eq}. We get from the inequality \eqref{eq:transversality}, DCT and MCT that
\begin{align}\label{eq:optimal-v-2}
 v(x,z,m)&\leq \mathbb{E}\left[\int_0^{\infty}e^{-\rho t}U(c_t^*)dt-\beta\int_0^{\infty}e^{-\rho t}dL_t^{X^*}\right]\nonumber\\
 &\leq \sup_{(\theta,c)\in\mathbb{U}^r}\mathbb{E}\left[\int_0^{\infty}e^{-\rho t}U(c_t)dt-\beta\int_0^{\infty}e^{-\rho t}dL_t^{X}\right].
\end{align}
Combining \eqref{eq:optimal-v-1} with \eqref{eq:optimal-v-2}, we conclude that the inequality in \eqref{eq:optimal-v-2} holds as an equality.
\end{proof}

\begin{proof}[Proof of Lemma \ref{lem:inject-captial}]
We first prove the item (i). For $(\mathrm{v},z,m)\in\R_+^3$, by applying \eqref{eq:w}, we have 
\begin{align}
\Ex\left[\int_0^{\infty} e^{-\rho t}dA^*_t \right]=\Ex\left[ \int_0^{\infty} e^{-\rho t} U(c_t^*)dt\right]-v(x,z,m),
\end{align}
with $x=({\rm v}-z)^+$ by Lemma \ref{lem:equivalence}. 
By uisng the dual relationship, we have that
\begin{align}\label{eq:dual-consumption}
U(c^*(x,z,m))=\frac{(c^*(x,z,m))^p}{p}=\frac{1}{p}\times
\begin{cases}
\displaystyle \lambda^p m^p, &(\lambda m)^{p-1}<y\leq \beta,\\[0.7em]
      \displaystyle y^{\frac{p}{p-1}}, &m^{p-1}<y\leq (\lambda m)^{p-1},\\[0.7em]
      \displaystyle m^p, & y^*(m)\leq  y\leq m^{p-1},\\[0.7em]
      \displaystyle (m^*(y))^p, & y<y^*(m),
\end{cases}
\end{align}
where $y=y(x,z,m)=v_x(x,z,m)$ and $y\to m^*(y)$ is the inverse function of $m\to y^*(m)$. From \eqref{eq:dual-consumption}, we can deduce that $U(c^*(x,z,m))\leq  y^{\frac{p}{p-1}}/|p|$ for all $(x,z,m)\in\R_+^3$. Then, it follows from  Lemma \ref{lem:Yt} and the proof of Lemma \ref{lem:transversality} that
\begin{align}\label{eq:optimal-c}
&\Ex\left[\int_0^{\infty} e^{-\rho t} U(c_t^*)dt\right]\leq \frac{1}{|p|}\Ex\left[\int_0^{\infty} e^{-\rho t} Y_t^{\frac{p}{p-1}} dt\right]<+\infty.
\end{align}
 Then, the desired result  \eqref{eq:dAstarfinite} follows from \eqref{eq:optimal-c}.

Next, we prove the item (ii). For any admissible portfolio $\theta=(\theta_t)_{t\geq0}$, we introduce, for all $t\geq0$,
\begin{align}\label{eq:wealth-theta-1}
\tilde{V}_t^{\theta} &=\textrm{v}+\int_0^t\theta_s^{\top}\mu ds+\int_0^t \theta_s^{\top}\sigma dW_s,\quad
\tilde{A}_t^{\theta}=0\vee \sup_{s\leq t}(Z_{s}-\tilde{V}^{\theta}_{s}).
\end{align}
Note that $c^*_t>0$ for all $t\geq0$. It follows from \eqref{eq:wealth2} and \eqref{eq:wealth-theta-1} that $\tilde{V}_t^{\theta^*}\geq V^{\theta^*,c^*}_t$ for all $t\in\R_+$, and hence
\begin{align}\label{eq:tilde-w}
\Ex\left[\int_0^{\infty} e^{-\rho t}dA^*_t \right]>\Ex\left[\int_0^{\infty} e^{-\rho t}d\tilde{A}_t^{\theta^*} \right]\geq \inf_{\theta}\Ex\left[\int_0^{\infty} e^{-\rho t}d\tilde{A}_t^{\theta} \right]=:\tilde{\mathrm{w}}(\mathrm{v},z).
\end{align}
It is not difficult to check that, for all $(\mathrm{v},z)\in\R_+\times(0,\infty)$,
\begin{align}\label{eq:tilde-v}
\tilde{\mathrm{w}}(\mathrm{v},z)=z\frac{1-\kappa}{\kappa}\left(1+\frac{(\mathrm{v}-z)^+}{z}\right)^\frac{\kappa}{\kappa-1}.
\end{align}
On the other hand, for any admissible portfolio $\theta=(\theta_t)_{t\geq0}$, let us consider an auxiliary process
\begin{align}\label{eq:wealth-theta-2}
\hat{V}_t^{\theta} &=\textrm{v}+\int_0^t\theta_s^{\top}\mu ds+\int_0^t \theta_s^{\top}\sigma dW_s- \lambda m t, \quad
\hat{A}_t^{\theta}=0\vee \sup_{s\leq t}(-\tilde{V}^{\theta}_{s}),\quad t\geq 0.
\end{align}
Note that $c^*_t\geq \lambda m$ for all $t\in\R_+$. Then, it follows from \eqref{eq:wealth2} and \eqref{eq:wealth-theta-2} that $\hat{V}_t^{\theta^*}\geq V^{\theta^*,c^*}_t$ for all $t\in\R_+$, and hence
\begin{align}\label{eq:hat-w}
\Ex\left[\int_0^{\infty} e^{-\rho t}dA^*_t \right]>\Ex\left[\int_0^{\infty} e^{-\rho t}d\hat{A}_t^{\theta^*} \right]\geq \inf_{\theta}\Ex\left[\int_0^{\infty} e^{-\rho t}d\hat{A}_t^{\theta} \right]=:\hat{\mathrm{w}}(\mathrm{v},m).
\end{align}
In a similar fashion, we can verify that, for all $(\mathrm{v},m)\in\R_+\times(0,\infty)$,
\begin{align}\label{eq:hat-v}
\tilde{\mathrm{w}}(\mathrm{v},m)=\frac{\lambda m}{\alpha+\rho}e^{-\frac{\alpha+\rho}{\lambda m}\mathrm{v}^+}.
\end{align}
Consequently, combining \eqref{eq:tilde-w}, \eqref{eq:tilde-v}, \eqref{eq:hat-w} and \eqref{eq:hat-v},  we complete the proof of the lemma.
\end{proof}

\begin{proof}[Proof of Corollary \ref{coro:no-constraint}]
Denote by $v_{\lambda}(x,z,m)$ the value function and $\mathbb{U}^{\rm r}_{\lambda}$ the admissible set to highlight the dependence on $\lambda$. Then,  for $0\leq \lambda_1\leq \lambda_2\leq 1$, it follows from the definition of the admissible set that $\mathbb{U}^{\rm r}_{\lambda_2} \subset \mathbb{U}^{\rm r}_{\lambda_1}$, which yields  $v_{\lambda_2}(x,z,m)\leq v_{\lambda_1}(x,z,m)$ for all $(x,z,m)\in\R_+^3$. 

When $\lambda=0$, by Proposition \ref{prop:sol-v}, we have $y^*(m)= m^{1-p}$ for $m\geq \beta^{1/(p-1)}$, and for $(r,z,m)\in\R_+\times [\beta^{1/(p-1)},\infty)$, the dual function $\hat{v}(y,z,m)$ becomes
\begin{align}\label{u-non-constraint}
\hat{v}(y,z,m)= \frac{1}{\beta}C_3(m)y+\beta^{\frac{\rho}{\alpha}}C_4(m)y^{-\frac{\rho}{\alpha}}+\frac{(1-p)^3}{p(\rho(1-p)-\alpha p)}y^{\frac{p}{p-1}}+z\left( y-\frac{\beta^{-\kappa+1}}{\kappa}y^{\kappa}\right)
\end{align}
Moreover, in view of $y^*(m)= m^{1-p}$, it holds that
\begin{align*}
C_6(m)
=\frac{\alpha^3\beta^{-\frac{\rho}{\alpha}}}{\rho(\alpha+\rho)^2(\rho(1-p)-\alpha p)}m^{\frac{\alpha p-(1-p)\rho}{\alpha}}.
\end{align*}
As a result, we can deduce that 
\begin{align}\label{eq:C3C4}
C_3(m)=\frac{(1-p)^2}{\rho(1-p)-\alpha p}\beta^{\frac{p}{p-1}},\quad C_4(m)=0,\quad \forall m\geq \beta^{\frac{1}{p-1}}.
\end{align}
It follows from  \eqref{u-non-constraint} and \eqref{eq:C3C4} that
\begin{align}\label{u-non-constraint-2}
\hat{v}(y,z,m)=\frac{(1-p)^2}{\rho(1-p)-\alpha p}\beta^{\frac{1}{p-1}}y+\frac{(1-p)^3}{p(\rho(1-p)-\alpha p)}y^{\frac{p}{p-1}}
+z\left( y-\frac{\beta^{-\kappa+1}}{\kappa}y^{\kappa}\right),
\end{align}
which is independent of the variable $m$. Then, by Proposition \ref{prop:sol-v}, Lemma \ref{lem:valuefunc-v}, Lemma \ref{lem:feedback-control} and Theorem \ref{thm:verification}, we get the desired results \eqref{valuefunc-non-constraint}-\eqref{feedbackcontrol-non-constraint}.
\end{proof}

\vspace{0.4in}
\noindent\textbf{Acknowledgements} The authors are grateful to two anonymous referees for their helpful comments and suggestions. L. Bo and Y. Huang are supported by National Natural Science of Foundation of China (No. 12471451), Natural Science Basic Research Program of Shaanxi (No. 2023-JC-JQ-05), Shaanxi Fundamental Science Research Project for Mathematics and Physics (No. 23JSZ010) and Fundamental Research Funds for the Central Universities (No. 20199235177). X. Yu is supported by the Hong Kong RGC General Research Fund (GRF) under grant no. 15306523, the Hong Kong Polytechnic University research grant under no. P0045654 and the Research Centre for Quantitative Finance at the Hong Kong Polytechnic University under grant no. P0042708.

\end{document}